\documentclass[%a4paper,
11pt,  reqno]{amsart}
\usepackage[margin=1.2in,marginparwidth=1.5cm, marginparsep=0.5cm]{geometry}

\usepackage{amsmath,amssymb,amscd,amsthm,amsxtra, esint}

\usepackage{soul}

\usepackage{booktabs} % nice tables
\usepackage{microtype}
\usepackage{amssymb}
\usepackage{mathrsfs}

\usepackage{upgreek}

\usepackage{color}
\usepackage[implicit=true]{hyperref}

\usepackage{xsavebox}

\usepackage{cases}%%%

\allowdisplaybreaks[2]

\definecolor{gr}{rgb}   {0.,   0.69,   0.23 }
\definecolor{bl}{rgb}   {0.,   0.5,   1. }
\definecolor{mg}{rgb}   {0.85,  0.,    0.85}
\definecolor{yl}{rgb}   {0.8,  0.7,   0.}
\definecolor{or}{rgb}  {0.7,0.2,0.2}

\usepackage{tikz}
\usepackage{marginnote}
\usepackage{scalerel} % to rescale the paraproduct symbols

\usetikzlibrary{shapes.misc}
\usetikzlibrary{shapes.symbols}
\usetikzlibrary{decorations}
\usetikzlibrary{decorations.markings}

\tikzset{
	dot/.style={circle,fill=black,draw=black,inner sep=0pt,minimum size=0.5mm},
	>=stealth,
	}

\tikzset{
	dot2/.style={circle,fill=black,draw=black,inner sep=0pt,minimum size=0.2mm},
	>=stealth,
	}

\tikzset{
	ddot/.style={circle,fill=black,draw=black,inner sep=0pt,minimum size=0.8mm},
	>=stealth,
	}

\tikzset{decision/.style={ % requires library shapes.geometric
        draw,
        diamond,
        aspect=1.5
    }}

\tikzset{dia2/.style
={diamond,fill=white,draw=black,inner sep=0pt,minimum size=1mm},
	>=stealth,
	}

\tikzset{dia/.style
={star,fill=black,draw=black,inner sep=0pt,minimum size=1mm},
	>=stealth,
	}

\tikzset{dia/.style
={diamond,fill=black,draw=black,inner sep=0pt,minimum size=1.3mm},
	>=stealth,
	}

\makeatletter
\def\DeclareSymbol#1#2#3{\xsavebox{#1}{\tikz[baseline=#2,scale=0.15]{#3}}}
\def\<#1>{\xusebox{#1}}
\makeatother

\newsavebox{\peA}
\newsavebox{\pneA}
\newsavebox{\plA}
\newsavebox{\pgA}
\newsavebox{\pleA}
\newsavebox{\pgeA}
\newsavebox{\pezA}

\savebox{\peA}{\tikz \draw (0,0) node[shape=circle,draw,inner sep=0pt,minimum size=8.5pt] {\scriptsize  $=$};}
\savebox{\pneA}{\tikz \draw (0,0) node[shape=circle,draw,inner sep=0pt,minimum size=8.5pt] {\footnotesize $\neq$};}
\savebox{\plA}{\tikz \draw (0,0) node[shape=circle,draw,inner sep=0pt,minimum size=8.5pt] {\scriptsize $<$};}
\savebox{\pgA}{\tikz \draw (0,0) node[shape=circle,draw,inner sep=0pt,minimum size=8.5pt] {\scriptsize $>$};}
\savebox{\pleA}{\tikz \draw (0,0) node[shape=circle,draw,inner sep=0pt,minimum size=8.5pt] {\scriptsize $\leqslant$};}
\savebox{\pgeA}{\tikz \draw (0,0) node[shape=circle,draw,inner sep=0pt,minimum size=8.5pt] {\scriptsize $\geqslant$};}
\savebox{\pezA}{\tikz \draw (0,0) node[shape=circle,draw,
fill=white, % color = white,
inner sep=0pt,minimum size=8.5pt]{} ;}

\def \peB{\mathchoice
{\scalebox{.7}{{\usebox{\peA}}}}
{\scalebox{.7}{{\usebox{\peA}}}}
{\scalebox{.7}{{\usebox{\peA}}}}
{}
}

\def \pezB{\mathchoice
{\scalebox{.7}{{\usebox{\pezA}}}}
{\scalebox{.7}{{\usebox{\pezA}}}}
{\scalebox{.7}{{\usebox{\pezA}}}}
{}
}

\newcommand{\pe}{\mathbin{{\peB}}}

\newcommand{\pez}{\mathbin{{\pezB}}}

\usetikzlibrary{decorations.pathmorphing, patterns,shapes}

\DeclareSymbol{dot}{-2.7}{\draw (0,0) node[dot]{};}

\DeclareSymbol{1}{0}{\draw[white] (-.4,0) -- (.4,0); \draw (0,0)  -- (0,1.2) node[dot] {};}
\DeclareSymbol{2}{0}{\draw (-0.5,1.2) node[dot] {} -- (0,0) -- (0.5,1.2) node[dot] {};}
\DeclareSymbol{3}{0}{\draw (0,0) -- (0,1.2) node[dot] {}; \draw (-.7,1) node[dot] {} -- (0,0) -- (.7,1) node[dot] {};}
\DeclareSymbol{30}{-3}{\draw (0,0) -- (0,-1); \draw (0,0) -- (0,1.2) node[dot] {}; \draw (-.7,1) node[dot] {} -- (0,0) -- (.7,1) node[dot] {};}
\DeclareSymbol{31}{-3}{\draw (0,0) -- (0,-1) -- (1,0) node[dot] {}; \draw (0,0) -- (0,1.2) node[dot] {}; \draw (-.7,1) node[dot] {} -- (0,0) -- (.7,1) node[dot] {};}
\DeclareSymbol{32}{-3}{\draw (0,0) -- (0,-1) -- (1,0) node[dot] {}; \draw (0,0) -- (0,-1) -- (-1,0) node[dot] {}; \draw (0,0) -- (0,1.2) node[dot] {}; \draw (-.7,1) node[dot] {} -- (0,0) -- (.7,1) node[dot] {};}
\DeclareSymbol{20}{-3}{\draw (0,0) -- (0,-1);\draw (-.7,1) node[dot] {} -- (0,0) -- (.7,1) node[dot] {};}
\DeclareSymbol{22}{-3}{\draw (0,0.3) -- (0,-1) -- (1,0) node[dot] {}; \draw (0,0.3) -- (0,-1) -- (-1,0) node[dot] {};\draw (-.7,1) node[dot] {} -- (0,0.3) -- (.7,1) node[dot] {};}
\DeclareSymbol{31p}{-3}{\draw (0,0) -- (0,-1) -- (1,0) node[dot] {}; \draw (0,0) -- (0,1.2) node[dot] {}; \draw (-.7,1) node[dot] {} -- (0,0) -- (.7,1) node[dot] {};
\draw (0,-1) node{\scaleobj{0.5}{\pez}};
\draw (0,-1) node{\scaleobj{0.5}{\pe}}; }
\DeclareSymbol{32p}{-3}{\draw (0,0) -- (0,-1) -- (1,0) node[dot] {}; \draw (0,0) -- (0,-1) -- (-1,0) node[dot] {}; \draw (0,0) -- (0,1.2) node[dot] {}; \draw (-.7,1) node[dot] {} -- (0,0) -- (.7,1) node[dot] {};
\draw (0,-1) node{\scaleobj{0.5}{\pez}};
\draw (0,-1) node{\scaleobj{0.5}{\pe}};}
\DeclareSymbol{22p}{-3}{\draw (0,0.3) -- (0,-1) -- (1,0) node[dot] {}; \draw (0,0.3) -- (0,-1) -- (-1,0) node[dot] {};\draw (-.7,1) node[dot] {} -- (0,0.3) -- (.7,1) node[dot] {};
\draw (0,-1) node{\scaleobj{0.5}{\pez}};
\draw (0,-1) node{\scaleobj{0.5}{\pe}};}
\DeclareSymbol{21p}{-3}{\draw (0,0.3) -- (0,-1) -- (1,0) node[dot] {};
\draw (-.7,1) node[dot] {} -- (0,0.3) -- (.7,1) node[dot] {};
\draw (0,-1) node{\scaleobj{0.5}{\pez}};
\draw (0,-1) node{\scaleobj{0.5}{\pe}};}

\DeclareSymbol{21}{-3}{\draw (0,0.3) -- (0,-1) -- (1,0) node[dot] {};
\draw (-.7,1) node[dot] {} -- (0,0.3) -- (.7,1) node[dot] {};
}

\DeclareSymbol{11p}{-3}{\draw (0,0) node[dot2] {} -- (0,-1) -- (1,0) node[dot] {};
\draw (0,0) -- (0,1.2) node[dot] {};  %\draw (-.7,1) node[dot] {} -- (0,0) -- (.7,1) node[dot] {};
\draw (0,-1) node{\scaleobj{0.5}{\pez}};
\draw (0,-1) node{\scaleobj{0.5}{\pe}}; }
\DeclareSymbol{100}{-3}
{\draw[white] (-.4,0) -- (.4,0);
\draw (0,0) node[dot2] {} -- (0,-1)  ;
\draw (0,0) -- (0,1.2) node[dot] {};}  %\draw (-.7,1) node[dot] {} -- (0,0) -- (.7,1) node[dot] {};
\def\R{\mathbb{R}}

\def\r2n{{\mathbb{R}^{2n}}}

\def\N{\mathbb{N}}

\def\Z{\mathbb{Z}}

\def\G{\mathscr{G}}

\def\supp{\operatorname{supp}}

\usetikzlibrary{mindmap,backgrounds,trees,arrows,decorations.pathmorphing,decorations.pathreplacing,positioning,shapes.geometric,shapes.misc,decorations.markings,decorations.fractals,calc,patterns}

\tikzset{>=stealth',
         cvertex/.style={circle,draw=black,inner sep=1pt,outer sep=3pt},
         vertex/.style={circle,fill=black,inner sep=1pt,outer sep=3pt},
         star/.style={circle,fill=yellow,inner sep=0.75pt,outer sep=0.75pt},
         tvertex/.style={inner sep=1pt,font=\scriptsize},
         gap/.style={inner sep=0.5pt,fill=white}}
\tikzstyle{mybox} = [draw=black, fill=blue!10, very thick,
    rectangle, rounded corners, inner sep=10pt, inner ysep=20pt]
\tikzstyle{boxtitle} =[fill=blue!50, text=white,rectangle,rounded corners]

\tikzstyle{decision} = [diamond, draw, fill=blue!20,
    text width=4.5em, text badly centered, node distance=2.5cm, inner sep=0pt]
\tikzstyle{block} = [rectangle, draw, fill=blue!20,
    text width=2.7cm, text centered, rounded corners, minimum height=3em]

\tikzstyle{line} = [draw, very thick, color=black!50, -latex']

\tikzstyle{cloud} = [draw, ellipse,fill=red!40,
node distance=2.5cm,
    minimum height=1em]

\tikzstyle{cloud2} = [draw, ellipse,fill=red!30, text=white,text width=10em, node distance=2.5cm, text centered, minimum height=4em]

\tikzstyle{cloud3} = [draw, ellipse, fill=cyan!30,
node distance=2.5cm,
    minimum height=3em,
    text width=1.7cm]

\tikzstyle{cloud4} = [draw, ellipse,fill=orange!70, node distance=2.5cm,
   text width=2.1cm,
           minimum height=3em]

\tikzstyle{cloud5} = [draw, ellipse,fill=red!20, node distance=2.5cm,
    text centered,
    text width=3.0cm,
    minimum height=3em]

\tikzstyle{cloud6} = [draw, ellipse,fill=red!20, node distance=2.5cm,
    text width=1.5cm,
    text centered,
    minimum height=3em]

\tikzset{
    position/.style args={#1:#2 from #3}{
        at=(#3.#1), anchor=#1+180, shift=(#1:#2)
    }
}
\newtheorem{theorem}{Theorem} [section]
\newtheorem{maintheorem}{Theorem}
\newtheorem{lemma}[theorem]{Lemma}
\newtheorem{proposition}[theorem]{Proposition}
\newtheorem{remark}[theorem]{Remark}

\newtheorem{definition}[theorem]{Definition}

\newtheorem{oldtheorem}{Theorem}

\DeclareMathOperator*{\intt}{\int}

\newcommand{\1}{\hspace{0.2mm}\text{I}\hspace{0.2mm}}
\newcommand{\II}{\text{I \hspace{-2.8mm} I} }
\newcommand{\III}{\text{I \hspace{-2.9mm} I \hspace{-2.9mm} I}}
\newcommand{\noi}{\noindent}
\newcommand{\T}{\mathbb{T}}

\newcommand{\al}{\alpha}

\newcommand{\dl}{\delta}

\newcommand{\Dl}{\Delta}
\newcommand{\eps}{\varepsilon}

\renewcommand{\G}{\Gamma}

\newcommand{\ft}{\widehat}

\newcommand{\wt}{\widetilde}
\newcommand{\cj}{\overline}

\newcommand{\dd}{\partial}

\renewcommand{\l}{\ell}

\newcommand{\les}{\lesssim}
\newcommand{\ges}{\gtrsim}

\newcommand{\jb}[1]
{\langle #1 \rangle}

\renewcommand{\S}{\mathcal{S}}

\newtheorem*{ackno}{Acknowledgements}

\numberwithin{equation}{section}
\numberwithin{theorem}{section}

\newcommand{\too}{\longrightarrow}

\usepackage{comment}

\DeclareMathOperator*{\ave}{ave}
\DeclareMathOperator{\diam}{diam}

\newcommand{\BMO}{\textit{BMO} }
\newcommand{\CMO}{\textit{CMO} }
\newcommand{\VMO}{\textit{VMO} }

\newcommand{\CZ}{Calder\'on-Zygmund }

\makeatletter
\DeclareRobustCommand\widecheck[1]{{\mathpalette\@widecheck{#1}}}
\def\@widecheck#1#2{%
   \setbox\z@\hbox{\m@th$#1#2$}%
   \setbox\tw@\hbox{\m@th$#1%
      \widehat{%
         \vrule\@width\z@\@height\ht\z@
         \vrule\@height\z@\@width\wd\z@}$}%
   \dp\tw@-\ht\z@
   \@tempdima\ht\z@ \advance\@tempdima2\ht\tw@ \divide\@tempdima\thr@@
   \setbox\tw@\hbox{%
      \raise\@tempdima\hbox{\scalebox{1}[-1]{\lower\@tempdima\box\tw@}}}%
   {\ooalign{\box\tw@ \cr \box\z@}}}
\makeatother

\makeatletter
\@namedef{subjclassname@2020}{%
  \textup{2020} Mathematics Subject Classification}
\makeatother

\begin{document}
\baselineskip = 14pt

\title[Compact $T(1)$ theorem \`a la Stein]
{Compact $T(1)$ theorem \`a la Stein}

\author[\'A. B\'enyi, G. Li, T. Oh, and R. H. Torres]
{\'Arp\'ad B\'enyi, Guopeng Li, Tadahiro Oh, and Rodolfo H. Torres}

\address{\'Arp\'ad B\'enyi, Department of Mathematics\\
516 High St, Western Washington University\\ Bellingham, WA 98225,
USA.}

\email{benyia@wwu.edu}

\address{
Guopeng Li, School of Mathematics and Statistics\\
Beijing Institute of Technology\\
Beijing\\
100081\\
China\\
and School of Mathematics\\
The University of Edinburgh\\
and The Maxwell Institute for the Mathematical Sciences\\
James Clerk Maxwell Building\\
The King's Buildings\\
Peter Guthrie Tait Road\\
Edinburgh\\
EH9 3FD\\
United Kingdom}

\email{guopeng.li@bit.edu.cn}

\address{
Tadahiro Oh, 
%School of Mathematics and Statistics\\
%Beijing Institute of Technology\\
%Beijing\\
%100081\\
%China\\
%and 
School of Mathematics\\
The University of Edinburgh\\
and The Maxwell Institute for the Mathematical Sciences\\
James Clerk Maxwell Building\\
The King's Buildings\\
Peter Guthrie Tait Road\\
Edinburgh\\
EH9 3FD\\
United Kingdom}

\email{hiro.oh@ed.ac.uk}

\address{Rodolfo H. Torres, Department of Mathematics\\
University of California\\
900 University Ave., Riverside, CA 92521, USA}

\email{rodolfo.h.torres@ucr.edu}



\begin{abstract}
We prove a compact $T(1)$ theorem, involving quantitative estimates, analogous
to the quantitative classical $T(1)$ theorem due to Stein. We also discuss the
$C_c^\infty$-to-\CMO mapping properties of non-compact \CZ operators as well as the sequential completeness properties of some subspaces of \BMO under different topologies.
%As a byproduct of our analysis,
%we present a direct proof  that a compact \CZ operator
%is compact   on the Hardy space $H^1$
%%and on $\BMO$
%%into itself
%under the cancellation assumption $T(1)=  T^*(1) =0$.

\end{abstract}

\maketitle

\tableofcontents

\section{Introduction}

The classical $T(1)$ theorem of David and Journ\'e \cite{DJ} is a fundamental result in harmonic analysis that characterizes the $L^2$-boundedness of operators with Calder\'on-Zygmund kernels.
A remarkable insight of this theorem is that one only needs to test whether the actions (properly defined) of the operator and its formal transpose on the constant function 1 belong to \BMO and to verify
the so-called weak boundedness property.

There are other equally useful versions of the $T(1)$ theorem that avoid the introduction of the weak boundedness property. For example, in \cite{DJ},  the $L^2$-boundedness is also shown to be equivalent to the uniform boundedness (with respect to $\xi$) in \BMO of the operator and its transpose acting on the character functions $x\mapsto e^{ix\cdot\xi}$.  In a very elegant formulation, Stein~\cite{Stein} proved a quantitative version of the $T(1)$ theorem that
entirely avoids \`a priori mentioning \BMO in the statement of the theorem  or defining the operator on smooth $L^\infty$-functions. He instead considered appropriate uniform $L^2$-estimates for the operator and its transpose acting on normalized bump functions; see Theorem \ref{THM:T1Stein} below. Although the actual proof of Stein's version does use the definition of $T(1)$ and equivalent formulations of the weak boundedness property, the advantage of his formulation is that, once the theorem has been established, the definitions of such concepts do not need to  be considered when applying the theorem to specific applications. See also~\cite{BO2} for a further discussion on this result by Stein.

Much more recently, several authors have considered versions of the $T(1)$ theorem
characterizing  the situation when operators with Calder\'on-Zygmund kernels are not only bounded but also compact.
The quest for this characterization aims to further complete the \CZ theory and understand the role of $T(1)$ when an operator in such theory is better than bounded, but  it is not necessarily motivated by the need to prove compactness for particular examples.  However,  explicit compact \CZ operators do exist.
Beyond some known boundary integral operators on bounded $C^1$ domains\footnote{Such operators are also of \CZ type on Lipschitz domains, but the extra $C^1$-smoothness is generally needed for compactness.}, whose compactness is nowadays well-understood,
an interesting example is provided by a certain class of pseudodifferential operators of order zero in $\R^d$ originally introduced by Cordes \cite{Cor}. More recently, in  \cite{CST}, such operators were re-brought to light for a version of the result in \cite{Cor} even in a weighted setting; see the definition of such operators in Subsection~ \ref{example} below. Other examples are provided by appropriate paraproducts, but such operators are usually employed in the proof of compact versions of the $T(1)$ theorem
and thus their compactness need to be proved by different methods.

% In  \cite{MitSto}, Mitkovski and Stockdale proved an  $L^2$-compactness  version of the $T(1)$ theorem by replacing the weak boundedness property in the classical $T(1)$ theorem  with an appropriate weak compactness property coupled with the stronger requirement that $T(1)$ and $T^*(1)$ belong to \CMO $\subsetneq$ \BMO (\CMO denotes  the space of functions of {\it continuous mean oscillation}; see Definition \ref{DEF:2}
%below).
% The \CMO space is a celebrity in harmonic analysis in its own right, appearing naturally in compactness results ever since the classical
%compactness
% result of Uchiyama \cite{Uch} on commutators of Calder\'on-Zygmund operators with \CMO functions. For a  multilinear version of the compact $T(1)$ theorem, see the recent work of Fragkos, Green, and Wick \cite{FraGreWic}.
% It is worth mentioning that the results in \cite{MitSto, FraGreWic} avoid the additional condition on the kernel of a Calder\'on-Zygmund operator that appeared in the first version of
% such a compact $T(1)$ theorem  due to
% Villarroya \cite{V}.

In \cite{V},  Villarroya proved a compactness version of the $T(1)$ theorem by
 replacing the weak boundedness property
and the  condition $T(1), T^*(1) \in \BMO$
in the classical $T(1)$ theorem
(see Theorem \ref{THM:DJ})
with an appropriate weak compactness property
and   the stronger
requirement that $T(1)$ and $T^*(1)$ belong to
\CMO $\subsetneq$ \BMO (\CMO denotes  the space of functions of {\it continuous mean oscillation}; see Definition \ref{DEF:2}
below),
and by
 imposing some additional conditions on the kernel of the operator. %a \CZ operator.
 More recently, Mitkovski and Stockdale \cite{MitSto}
  in the linear case and Fragkos, Green, and Wick \cite{FraGreWic} in the multilinear setting, established compactness versions of the $T (1)$ theorem using a more general formulation of the weak compactness property but avoiding the additional conditions on the kernel of
  a \CZ  operator that appeared in the original version of Villarroya;
  see also a recent work \cite{CLSY}
%by
%Cao, Liu, Si, and Yabuta
for a version of the bilinear compact $T (1)$ theorem
on weighted Lebesgue spaces
(but with additional  conditions on the kernel of
  a bilinear \CZ  operator).
We note that the space \CMO is a celebrity in harmonic analysis in its own right, appearing naturally in compactness results
ever since the classical
compactness
 result of Uchiyama \cite{Uch} on commutators of Calder\'on-Zygmund operators with \CMO functions.

The main goal of this article is to present a version of a compact $T(1)$ theorem
in the spirit of Stein's $T(1)$ theorem on boundedness (Theorem \ref{THM:T1Stein}),
without any reference to the actions of $T$ and $T^*$ on the constant function $1$ in the statement of the result.
A secondary goal of this article is to discuss the $C_c^\infty$-to-\CMO mapping properties of non-compact Calder\'on-Zygmund operators, both of convolution and non-convolution types, thus filling a gap (and also fixing some incorrect statements) in the literature.
For this purpose,we also revisit some delicate topological properties of $\BMO$\,-type spaces;
see Appendix \ref{SEC:weak}.

This article is organized as follows.
In Section \ref{SEC:linear},  we go over some background  and definitions needed for the remainder of the paper.
After recalling
 the classical $T(1)$ theorem and its compact counterpart
 in Subsection \ref{SUBSEC:T1a},
 we state our main result, a compact $T(1)$ theorem \`a la Stein (Theorem \ref{THM:1})
 in Subsection  \ref{SUBSEC:T1b}.
Section \ref{SEC:4} is dedicated to the proof of Theorem~\ref{THM:1}.
In Section \ref{SEC:5}, we then consider the $C_c^\infty$-to-\CMO mapping properties
of non-compact  Calder\'on-Zygmund operators.
As a byproduct of our analysis,
we also present a direct proof  %(namely without duality)
that,
under the cancellation assumption $T(1)=  T^*(1) =0$,
 a compact \CZ operator
is compact   on the Hardy space $H^1$; see Proposition~\ref{PROP:cpt}.
In the appendices,
we discuss
non-compactness  of convolution operators
as well as
the fact about \BMO and some of its subspaces not being weakly sequentially  complete,
 which are of interest on their own.
The results in the appendices are probably known to a more specialized functional analysis audience or may follow from more abstract facts. Hence, we do not want to claim authorship on them but they are included here for completeness. We provide references when we could locate them in the literature,
but when we could not, we
present some essentially self-contained arguments for some results for the benefit of readers.

\section{Calder\'on-Zygmund operators, {\it BMO}, and \CMO}\label{SEC:linear}

A  linear  singular integral operator $T$ is a map,
a priori defined and continuous from $\S(\R^d)$ into $\S'(\R^d)$ (the usual Schwartz space and its dual),
 that takes the form
\begin{align}
T(f)(x) = \int_{\R^d} K(x, y) f(y) dy
\label{CZ1}
\end{align}

\noi
when the point $x$ is not in the support of $f$.
Here,  we assume  that, away from the diagonal
 $\Dl = \{ (x, y) \in \R^{2d}:\, x = y\}$,
 the distributional kernel $K$ of $T$ coincides with
a function that is locally integrable on $\R^{2d} \setminus \Dl$, which we still call $K$.
The formal transpose $T^*$ of $T$
is defined similarly
with
the kernel  $K^*$
given by
$K^*(x, y)  = K(y, x)$.

\begin{definition}
\rm

A locally integrable function $K$ on $ \R^{2d} \setminus \Dl$
is called a {\it \CZ  kernel}
if it satisfies the following two conditions:

\begin{itemize}
\item[\textup{(i)}]
For all $x, y \in \R^d$, we have
\begin{align}
|K(x, y)|  \les |x-y|^{-d}.
\label{CZ1a}
\end{align}

\smallskip

\item[\textup{(ii)}]
There exists $\dl \in (0, 1]$ such that
\begin{align}
|K(x, y) - K(x', y)|
& + |K(y, x) - K(y, x')|
\les \frac{|x - x'|^\dl}{|x-y|^{d+\dl}}
\label{CZ2}
\end{align}

\noi
for all $x, x', y \in \R^d$
satisfying
$|x - x'| < \frac 12 |x-y|$.

\end{itemize}

\end{definition}

We say that a linear operator $T$
of the form \eqref{CZ1}
with a \CZ kernel $K$
is a {\it  \CZ operator}  if $T$ extends
to a bounded operator on $L^{p_0}(\R^d)$
for some $1< p_0 < \infty$.
It is well-known \cite{Stein1} that
if  $T$ is a  \CZ operator,
then it is bounded on $L^p(\R^d)$ for all
$1< p < \infty$. Hence, in the following, we restrict our attention to the $L^2$-boundedness of such linear operators. It is also well  known  that a \CZ operator is
bounded from
$L^\infty(\R^d)$ to $\BMO\,(\R^d)$ and
also from $H^1(\R^d)$ to $L^1(\R^d)$.
Here,  $H^1(\R^d)$ denotes the Hardy space,
which is  the  predual of $\BMO$ \cite{FS},
and
 $\BMO$ denotes the space of functions of \emph{bounded mean oscillation}, which we now recall.

\begin{definition} \label{DEF:2} \rm
Given a locally integrable function $f$ on $\R^d$,
define the $\BMO$\,-seminorm by
\begin{align}
 \|f\|_{\BMO} =
\sup_Q M_Q(f) =
 \sup_{Q} \frac{1}{|Q|} \int_Q|f(x) - \ave_Q f|\, dx,
\label{M1}
\end{align}
where the supremum is taken over all cubes $Q$ in $ \R^d$ with sides parallel to the coordinate axes,\footnote{In the following, it is understood that  all the cubes  have  sides parallel to the coordinate axes.}
and
\begin{align}
\ave_Q f  = \frac{1}{|Q|} \int_Q f(x) dx.
\label{ave1}
\end{align}
Then, we say that $f$ is of bounded mean oscillation if $\|f\|_{\BMO}< \infty$
and
we define $\BMO\,(\R^d)$ by
\[ \BMO\,(\R^d)  =
\big\{ f \in L^1_{\text{loc}}(\R^d):\,
\| f \|_{\BMO} < \infty \big\}.\]
As usual, the elements of this space need to be considered as equivalent classes of functions modulo additive constants.

We denote by $C^\infty_c(\R^d)$ (and by $C_0(\R^d)$, respectively)
 the space of $C^\infty$-functions with compact supports
(continuous functions vanishing at infinity, respectively).
  Also, we denote by $(C^\infty_c)_0(\R^d)$
the subspace of $C^\infty_c(\R^d)$  consisting of mean-zero functions.
We will often suppress the underlying space $\R^d$ from our notation.

The closure of $C_c^\infty$ in the $\BMO$ topology is called the space of functions of \emph{continuous mean oscillation}, and it is denoted by $\CMO$\,. By definition \CMO is a closed subspace of $\BMO$\,,
where the topology of \CMO is  inherited from that of $\BMO$\,, induced by the \BMO semi-norm in \eqref{M1}.
%Lastly, we recall that the Hardy space $H^1$ is the  predual of $\BMO$\,,
%while $H^1$ is the  dual of $\CMO$\,;
%see \cite{Stein1}

\end{definition}

We also recall the following characterization of the Hardy space (see, for example, \cite[Corollary~1 on p.\,221]{Stein1}).
Let $d \ge 2$.\footnote{When $d = 1$,
we replace the Riesz transforms by the Hilbert transform.}
Given $j = 1, \dots, d$, let $R_j$ be the Riesz transform
defined by $\ft {R_j(f)}(\xi) = i \frac{\xi_j}{|\xi|} \ft f(\xi)$,
 where $\widehat {~}$ denotes the Fourier transform.
Then, the Hardy space $H^1$
consists of functions $f \in L^1$
 such that  $R_j f \in L^1$, $j = 1, \dots, d$,
endowed with the norm:
\begin{align}
\|f||_{H^1} = \|f\|_{L^1} + \sum_{j=1}^{d} \| R_jf\|_{L^1}.
\label{H1}
\end{align}

Next, we recall a characterization of
 $\CMO$\,; see \cite[Lemma on p.\,166]{Uch}. We follow here Bourdaud's formulation in \cite[Th\'eor\`eme 7\,(iii) on p.\,1198]{Bour}. Given $f \in L^1_{\text{loc}}$,
we define
\begin{align}
\G_1(f) = \lim_{r \to 0}\bigg(
\sup_{\substack{|Q|  \le r}}
 \frac{1}{|Q|} \int_{Q} |f(x) - \ave_{Q} f|dx
\bigg),
\label{R1}
\end{align}

\noi
where the supremum is taken over all cubes
 of volume $|Q|\le r$
 and $\displaystyle\ave_{Q}f$ is as in~\eqref{ave1}.
Similarly, we define
\begin{align}
\G_2(f) = \lim_{r \to \infty}\bigg(
\sup_{|Q| \ge r}
 \frac{1}{|Q|} \int_{Q} |f(x) - \ave_{Q} f|dx
\bigg),
\label{R2}
\end{align}

\noi
where the supremum is taken over all cubes
 of volume $|Q| \ge r$.
Lastly, given a cube $Q\subset \R^d$,
we define
\begin{align}
\G_3^Q(f) = \limsup_{\substack{x_0 \in \R^d\\|x_0| \to \infty}}\bigg(
 \frac{1}{|Q|} \int_{Q+x_0} |f(x) - \ave_{Q+x_0} f|dx
\bigg).
\label{R3}
\end{align}

\noi
Then,
we have the following characterization
of $\CMO$
by the quantities $\G_1$, $\G_2$, and $\G_3^Q$.

\begin{lemma}\label{LEM:CMO}
Let $f \in \BMO$.
Then, $f \in \CMO$
if and only if
\begin{align}
\G_1(f) = \G_2(f) = \G_3^Q(f) = 0
\label{R4}
\end{align}

\noi
for any cube $Q \subset \R^d$.

\end{lemma}

In \cite{Bour},
Lemma \ref{LEM:CMO}
was stated in terms of
the analogues of
the quantities $\G_1$, $\G_2$, and $\G_3^Q$,
where cubes are replaced by
 balls.
A straightforward modification of the argument in~\cite{Bour}
yields Lemma \ref{LEM:CMO},
where we work with cubes.

\section{{\it T}(1) theorems and the main result}
\label{SEC:T1}

\subsection{Classical and compact {\it T}(1) theorems}
\label{SUBSEC:T1a}

In this subsection, we provide a brief discussion on the classical
$T(1)$ theorem (Theorem \ref{THM:DJ})
and its compact counterpart (Theorem \ref{THM:MS}).
 In order to do so, we need a few more definitions.

\begin{definition}\rm
We say that a function $\phi \in C_c^\infty$ is
a \emph{normalized bump function} of order $M$ if
$\supp \phi \subset B_0(1)$
and $\|\dd^\al \phi\|_{L^\infty} \leq 1$
for all multi-indices $\al$ with $|\al| \leq M$.
Here, given $r > 0$ and $x \in \R^d$,
  $B_x(r)$ denotes the ball of radius $r$ centered at $x$.

\end{definition}

Given $x_0 \in \R^d$ and $R>0$, we set
\begin{align}
\phi^{x_0, R}(x) = \phi\Big(\frac{x-x_0}{R}\Big).
\label{scaling}
\end{align}

\noi
The statement of the $T(1)$ theorem of David and Journ\'e \cite{DJ} is the following.

\begin{oldtheorem}[$T(1)$ theorem]\label{THM:DJ}
Let $T: \S \to \S'$ be a linear singular integral  operator
with a %standard
Calder\'on-Zygmund kernel.
Then, $T$ can be extended to a bounded operator on $L^2$
if and only if
it satisfies the following two conditions\textup{:}

\smallskip

\begin{itemize}
\item[\textup{(A.i)}]
$ T $ satisfies the weak boundedness property\textup{;}
there exists $M \in \mathbb N\cup \{0\}$
such that we have
\begin{align}
 \big| \big\langle T(\phi_1^{x_1, R}), \phi_2^{x_2, R}\big\rangle \big|\les R^d
\label{WBP1}
 \end{align}

\noi
for any normalized bump functions $\phi_1$ and $\phi_2$
of order $M$, $x_1, x_2 \in \R^d$,
and $R>0$.

\smallskip
\item[\textup{(A.ii)}]
$ T(1)$
and $T^*(1)$
are in $\BMO$\,.
\end{itemize}

\end{oldtheorem}

Since $T$ is a priori defined only in $\S$,
the expressions $T(1)$ and $T^*(1)$ need to be interpreted carefully;
see \eqref{X4} below. %\cite{DJ}.

\begin{definition} \label{DEF:WCP1} \rm

We say that 	
a linear singular integral  operator  $T: \S\to \S'$ with a Calder\'on-Zygmund kernel
has the {\it weak compactness property}
if there exists $M \in \mathbb N\cup \{0\}$
such that we have
\begin{align}
\lim_{|x_0|  + R + R^{-1} \to \infty}
R^{-d} \big| \big\langle T(\phi_1^{x_0 + x_1, R}), \phi_2^{x_0 + x_2, R}\big\rangle \big|
= 0
\label{WCP1}
 \end{align}

\noi
for any normalized bump functions $\phi_1$ and $\phi_2$
of order $M$ and  $x_1, x_2 \in \R^d$.

\end{definition}

Theorem \ref{THM:MS}, which we state below, is a compact counterpart of the $T(1)$ theorem (Theorem~\ref{THM:DJ}) and was recently proved in \cite{MitSto}. The original characterization of compactness of \CZ operators in \cite{V} states (a version of) the weak compactness property \eqref{WCP1} only in the ``diagonal case", that is with $x_1=x_2$, but imposes instead an additional behavior on the kernels of the operators. For a version of a compact $T(b)$ theorem, see \cite{V2}.

\begin{oldtheorem}[compact $T(1)$ theorem]\label{THM:MS}
Let $T$ be as in Theorem \ref{THM:DJ}.
Then, $T$ can be extended to a compact operator on $L^2$
if and only if
it satisfies the following two conditions\textup{:}

\smallskip

\begin{itemize}
\item[\textup{(B.i)}]
$ T $ satisfies the weak compactness  property.

\smallskip
\item[\textup{(B.ii)}]
$ T(1)$
and $T^*(1)$
are in $\CMO$\,.
\end{itemize}

\end{oldtheorem}

We say that  a linear singular integral  operator $T$
with a %standard
Calder\'on-Zygmund kernel
is a {\it compact \CZ operator} if it is compact on $L^2$.

\begin{remark}\label{REM:CMO} \rm
 If $T$ is a compact \CZ operator, % $T$ is compact on $L^2$,
then it is also compact

\smallskip

\begin{itemize}
\item[(a)]
on
$L^p$ for any $1 < p < \infty$;
this follows from interpolation
of compactness
 \cite{Kra},

\smallskip

\item[(b)]
from $L^\infty$ to $\CMO$\,; see
\cite{PPV},
%\cite[Theorem 2.18]{PPV}.

\smallskip

\item[(c)]
from the Hardy space $H^1$ to $L^1$;
this follows from (b) applied to $T^*$
(which is compact on $L^2$ in view of Lemma \ref{LEM:com}\,(ii))
and duality;\footnote{Although $L^1$ is not the dual of $L^\infty$,
an argument using duality can still be applied because the $L^1$-norm can be computed by pairing with
$L^\infty$-functions.}
 see~\cite[p.\,1285]{OV},

\smallskip

\item[(d)]
from $L^1$ to $L^{1, \infty}$; see~\cite{OV},

\smallskip

\item[(e)]
 from $\BMO$ to $\CMO$
 (\cite{PPV})
 under the cancellation assumption that $T(1) = T^*(1) = 0$,

\smallskip

\item[(f)]
 from $ H^1$ to $H^1$
 under the cancellation assumption that $T(1) = T^*(1) = 0$;
 this claim follows from (e) and duality;
 see also   Proposition~\ref{PROP:cpt}.
\end{itemize}
\end{remark}

\begin{remark} \rm
Our definitions of the weak boundedness property \eqref{WBP1}
and the weak compactness property (Definition \ref{WCP1})
are slightly different from those in \cite{DJ, MitSto},
but one can easily check that they are equivalent.

We also note that, as for the weak boundedness property,
 it suffices to verify \eqref{WBP1} for $x_1 = x_2$;
see \cite{Hart1}.
This does not seem to be the case
for the weak compactness property~\eqref{WCP1}; hence, the need for the additional conditions on the kernels imposed in the compactness version in \cite{V}. However, we are not aware if any of these conditions are fully independent, a point that has not been addressed in the literature and may be interesting to explore further. We do not plan to pursue such a task here, as ~\eqref{WCP1} suffices for our purposes.
\end{remark}

\subsection{{\it T}(1)  theorems \`a la Stein}
\label{SUBSEC:T1b}

We first recall
a quantitative formulation of the classical $T(1)$ theorem due to Stein;
see \cite[Theorem 3 on p.\,294]{Stein}.

\begin{oldtheorem}[$T(1)$ theorem \`a la Stein]\label{THM:T1Stein}
Let $T$ be as in Theorem \ref{THM:DJ}.
Then, $T$ can be extended to a bounded operator
on $L^2$
if and only if
there exists $M \in \mathbb{N} \cup\{0\}$
such that
we have
\begin{align}
\label{S1}
  \|T(\phi^{x_0, R})\|_{L^2}+\|T^*(\phi^{x_0, R})\|_{L^2}\les R^\frac{d}{2}
\end{align}

\noi
for any normalized bump function $\phi$
of order $M$, $x_0 \in \R^d$,
and $R>0$.
\end{oldtheorem}

Here,
the conditions (A.i) and (A.ii) in Theorem \ref{THM:DJ}
are replaced by the quantitative estimate~\eqref{S1} involving normalized bump functions. The advantage of Stein's version of the $T(1)$ theorem is that it completely avoids computing the action of the operator on the constant function 1, which in the literature is usually done only for some classical examples at a very formal level as it is very tedious to compute rigorously. It is natural to ask then whether a similar situation occurs in the study of compactness of \CZ operators.
In the following, we seek
for a suitable replacement of
the condition \eqref{S1}
for a compact $T(1)$ theorem \`a la Stein.
Before proceeding further, we first recall the following basic properties
of compact operators.

\begin{lemma}\label{LEM:com}
Let $X$ and $Y$ be Banach spaces
and $T:X\to Y$ be a continuous linear operator.

\smallskip

\begin{itemize}
\item[\textup{(i)}]
If $T$ is compact,
then $T$ maps  weakly convergent sequences
to strongly convergent sequences.

\smallskip

\item[\textup{(ii)}]
The operator
$T$ is compact if
and only if its transpose $T^*$ is compact.

\end{itemize}

\end{lemma}

Lemma \ref{LEM:com} is well known
and its proof can be  found in many textbook covering the subject;
see, for example,   Reed and Simons \cite[Theorems VI.11
and  VI.12 (c)]{RS}.

We now derive  the $L^2$- and $L^\infty$-conditions
for our formulation of a compact $T(1)$ theorem \`a la Stein.

\medskip

\noi
\underline{\bf $\pmb {L^2}$-condition:}
We first rewrite  the condition \eqref{S1}
 as
\begin{align*}
  \|T(R^{-\frac d2}\phi^{x_0, R})\|_{L^2}+\|T^*(R^{-\frac d2}\phi^{x_0, R})\|_{L^2}\les 1.
\end{align*}

\noi
Note that
$R^{-\frac d2} \phi^{x_0, R}$
converges weakly to $0$ in $L^2$
as $|x_0| + R + R^{-1} \to \infty$.
Suppose now that $T$ is compact on $L^2$.
Then,
from Lemma \ref{LEM:com}\,(i),
we see that
$R^{-\frac d2} T(\phi^{x_0, R})$
and
$R^{-\frac d2} T^*(\phi^{x_0, R})$
converge strongly to $0$ in $L^2$
as $|x_0| + R + R^{-1} \to \infty$,
namely,
\begin{align}
\lim_{|x_0| + R + R^{-1} \to \infty}
\bigg(
R^{-\frac d2}  \|T(\phi^{x_0, R})\|_{L^2}
+
R^{-\frac d2}   \|T^*(\phi^{x_0, R})\|_{L^2}
\bigg) = 0.
\label{R0}
\end{align}

\medskip

\noi
\underline{\bf $\pmb {L^\infty}$-condition:}
Let $T$ be a \CZ operator that is compact on $L^2$.
From Remark~\ref{REM:CMO}\,(b), we see that
$T$ is compact from
$L^\infty$ to $\CMO$\,.
Let  $\{f_n\}_{n \in \N}$ be  a weak-$\ast$ convergent sequence
in $L^\infty$.\footnote{Recall that $L^\infty(\R^d) = (L^1(\R^d))^*$.}
Then,
it follows from the Banach-Steinhaus theorem that
the sequence
$\{f_n\}_{n \in \N}$ is bounded in $L^\infty$,
and, hence, from the compactness
of $T$ from $L^\infty$ to $\CMO$\,,
we see that
 $\{T(f_{n})\}_{n \in \N}$ is precompact in $\CMO$\,.
In view of Lemma \ref{LEM:com}\,(ii),
the same conclusion holds for $T^*$.
This leads  to the $L^\infty$-condition (ii) in Theorem \ref{THM:1} below.

\medskip

We are now ready to state our main result, a  compact counterpart of the $T(1)$ theorem \`a la Stein.

\begin{maintheorem}[compact $T(1)$ theorem \`a la Stein]\label{THM:1}
Let $T$ be as in Theorem \ref{THM:DJ}.
Then, $T$ can be extended to a compact operator
on $L^2$
if and only if
the following two conditions hold\textup{:}

\begin{itemize}
\item[(i)]
\textup{($L^2$-condition).}
There exists $M \in \mathbb{N} \cup\{0\}$
such that
\eqref{R0}
 holds for any normalized bump function $\phi$
of order $M$.

\smallskip

\item[(ii)]
\textup{($L^\infty$-condition).}
Given any
weak-$\ast$ convergent sequence $\{f_n\}_{n \in \N}$
in $L^\infty$
such that  $\{T(f_{n})\}_{n \in \N}$
and $\{T^*(f_{n})\}_{n \in \N}$
are bounded in $\CMO$\,,
both the sequences
  $\{T(f_{n})\}_{n \in \N}$
and $\{T^*(f_{n})\}_{n \in \N}$
are precompact in $\CMO$\,.

\end{itemize}

\end{maintheorem}

Note that in the derivation of the $L^\infty$-condition and the $L^2$-condition before the statement of Theorem \ref{THM:1},
 we have already established their necessity,  assuming that $T$ is compact.
 Hence, it remains to prove  their sufficiency. Note also that the $L^\infty$-condition does not per se imply that $T$ is compact from  $L^\infty$ to $\CMO$.

 Our proof of Theorem \ref{THM:1} follows the lines  in the proof of \cite[Theorem 1]{BO2}
on the classical $T(1)$ theorem \`a la Stein
for  boundedness of linear singular integral operators.
The main task is to
verify the conditions (B.i) and (B.ii)
in Theorem \ref{THM:MS}
by assuming (i) and (ii) in Theorem \ref{THM:1}.
In particular,
for verifying the condition (B.ii),
we make use of the characterization
of $\CMO$ functions
given in
Lemma \ref{LEM:CMO}
and the condition (ii) in Theorem \ref{THM:1}.

\begin{remark}\rm
 It is not clear at this point  if the $L^\infty$-condition (ii) in Theorem \ref{THM:1}
can be removed.
See Remark \ref{REM:VMO} below.
\end{remark}

\begin{remark}\rm
In view of a compact bilinear $T(1)$ theorem in \cite{FraGreWic} and other properties of compact bilinear operators recently studied in \cite{BLOT},
it would be of interest to investigate an analogue of
Theorem \ref{THM:1} in the bilinear setting. We leave this to  interested readers.
\end{remark}

\subsection{An application of Theorem 1} \label{example} We end this section by illustrating the applicability of Theorem \ref{THM:1} to a simple example of compact \CZ operator. See also \cite{FraGreWic} for a similar example. For an application of the compact $T(1)$ theorem in establishing the compactness of the commutator of a certain class of pseudodifferential operators of order one and multiplication by a bounded Lipschitz function, see the recent work \cite{BOT}.

Following \cite{CST}, consider the class of pseudodifferential operators of order zero of the form\footnote{We are using the definition of the Fourier transform given by $\widehat f (\xi) = \int_{\mathbb R^d} f(x) e^{-ix\cdot\xi} dx$}
\begin{align*}
T_\sigma(f)(x)=\int_{\mathbb R^d}\sigma (x, \xi)\widehat f(\xi) e^{ix\cdot\xi} d\xi,
%\label{T1}
\end{align*}
where the symbol $\sigma$ satisfies
\begin{equation}
\label{VCM}
|\partial_x^\alpha\partial_\xi^\beta\sigma (x, \xi)|\leq K_{\alpha, \beta}(x, \xi)(1+|\xi|)^{-|\beta|}
\end{equation}
 for all multi-indices $\alpha$ and  $\beta$,  and where $K_{\alpha, \beta}(x, \xi)$ is a bounded function such that
 \begin{equation}
\label{VCM2}
 \displaystyle\lim_{|x|+|\xi|\to\infty}K_{\alpha, \beta}(x, \xi)=0.
 \end{equation}
The collection of pseudodifferential symbols $\sigma$ satisfying \eqref{VCM} with \eqref{VCM2} is a subclass of the classical H\"ormander class $S^0_{1,0}$, which is well-known to yield \CZ operators $T_\sigma$. To prove the compactness of these operators using any version of the compact $T(1)$ theorem immediately faces the nontrivial task of computing their transposes. While a symbolic calculus exits, as recently developed in \cite{BOT}, and in particular showing that this class of operators with symbols satisfying \eqref{VCM} with \eqref{VCM2} is closed under transposition, we will only consider one of the simpler examples for which the transpose of the operator is trivially computed. Although the compactness of the following example may be obtained in a variety of ways, the purpose here is to illustrate the utility of Theorem \ref{THM:1}.

 Let $\sigma_{a,b}(x,\xi) = (2\pi)^{-d}a(x) \widehat b(\xi)$, where $a, b \in C_c^\infty$ are supported on $B_0(1)$. Trivially, $\sigma_{a,b}(x,\xi)$ satisfies \eqref{VCM} with \eqref{VCM2}, and we can write
 $$
 T_{\sigma_{a,b}}(f) = a(b * f).
$$
It then easily follows that
$$
T_{\sigma_{a,b}}^*(f) = b * (a f).
 $$

 Let $\phi^{x_0,R}$ be a normalized bump function (of any order).  Because of the supports of the functions involved we have that
 $$
\| T_{\sigma_{a,b}}(\phi^{x_0,R}) \|_{L^2}= 0
$$
for $|x_0|\gg 1$. Also,
\begin{align*}
\|T_{\sigma_{a,b}}(R^{-\frac d2}\phi^{x_0, R})\|^2_{L^2} &\lesssim
    \int_{|x|\leq 1} |a(x)|^2  \left( \int_{\R^d} |b(x-z)| |\phi^{x_0,R}(z)| R^{-d/2}dz\right)^2 dx\\
    &\lesssim R^{-d} \| a\|^2_{L^2}  \| b\|^2_{L^1} \to 0
\end{align*}
as $R\to \infty$, while
\begin{align*}
    \|T_{\sigma_{a,b}}(R^{-\frac d2}\phi^{x_0, R})\|^2_{L^2} &\lesssim
    \int_{|x|\leq 1} |a(x)|^2  \left( \int_{|z|\leq 1} |b(x-Rz-x_0)| |\phi(z)| R^{d/2}dz\right)^2 dx\\
 &\lesssim R^{d} \| a\|^2_{L^2}  \| b\|^2_{L^\infty} \to 0
\end{align*}
as $R \to 0$. It follows that the $L^2$-condition is satisfied for $T_{\sigma_{a,b}}$.

 Regarding  the $L^\infty$-condition, we simply observe that if $\{f_j\}$ is a bounded sequence in $L^\infty$ with $\{T_{\sigma_{a, b}}(f_j)\}$ bounded in $\CMO$ then, again because of the condition on the supports of $a$ and $b$,
 $\{T_{\sigma_{a,b}}(f_j)\}$ is a sequence of $C^\infty$ functions supported on $B_0(1)$ which is equibounded and equicontinuous. By the Arzel\`a–Ascoli theorem, there exists a subsequence of $\{T_{\sigma_{a,b}}(f_j)\}$ which converges in $C_0$, and a fortiori in $\CMO$. The computations for $T_{\sigma_{a,b}}^*$ are similar. By Theorem \ref{THM:1}, $T_{\sigma_{a,b}}$ is compact. %Incidentally, for this simple example, the argument above in fact proves directly that $T$ is also compact from $L^\infty$ to $\CMO$.

\section{Proof  of Theorem \ref{THM:1}}
\label{SEC:4}

In this section, we present a proof of Theorem \ref{THM:1}.
As mentioned above,
it suffices to prove sufficiency of the conditions (i) and (ii).
In the following,
by assuming the conditions (i) and~(ii) in Theorem \ref{THM:1},
we verify
the conditions (B.i) and  (B.ii) in Theorem \ref{THM:MS},
which will in turn imply that $T$ is compact on $L^2$.

We first prove (B.i): the weak compactness property.
Let $\phi_1$ and $\phi_2$  be normalized bump functions   of order 0.
From \eqref{scaling}, we have
\begin{align}
\|R^{-\frac d2}\phi_1^{x_0 + x_1, R}\|_{L^2}
+
\|R^{-\frac d2}\phi_2^{x_0 + x_2, R}\|_{L^2}\les 1,
\label{X3}
\end{align}

\noi
uniformly in $x_0, x_1, x_2 \in \R^d$
and $R > 0$.
Then, from \eqref{R0}
and  \eqref{X3}, we have
\begin{align*}
& \lim_{|x_0|+ R + R^{-1} \to \infty}
R^{-d} \big| \big\langle T ( \phi_1^{x_0+ x_1, R}), \phi_2^{x_0+ x_2, R}\big\rangle \big|\\
& \hphantom{XXXX}
 \le
 \lim_{|x_0|+ R + R^{-1} \to \infty}
\min\Big(  \|T ( R^{-\frac d2} \phi_1^{x_0+ x_1, R})\|_{L^2}\|R^{-\frac d2}\phi_2^{x_0+ x_2, R}\|_{L^2},
\\
& \hphantom{lXXXXXXXX}
  \| R^{-\frac d2} \phi_1^{x_0+ x_1, R}\|_{L^2}\|T^*(R^{-\frac d2}\phi_2^{x_0+ x_2, R})\|_{L^2}\Big)\\
& \hphantom{XXXX}
 = 0
 \end{align*}

\noi
for any $x_1, x_2 \in \R^d$.
This proves
the condition (B.i) in Theorem \ref{THM:MS}.

Next, we prove the condition (B.ii) in Theorem \ref{THM:MS}.
In the following, we only show $T(1) \in \CMO$\,,
since  $ T^*(1) \in \CMO$
follows analogously by symmetry.
Since \eqref{R0} implies the bound~\eqref{S1},
it follows from Theorem \ref{THM:T1Stein}
that $T(1) \in \BMO$\,,
where  $T(1)$ is interpreted
as
\begin{align}
\jb{T(1), g} =
\lim_{R \to \infty}
\jb{T(\phi_R), g}
\label{X4}
\end{align}

\noi
for any $g \in (C^\infty_c)_0$.\footnote{
Recall that $(C^\infty_c)_0$ is dense in the Hardy space $H^1$,
which is a predual of $\BMO$\,; see \cite[p.\,372]{DJ}.
Namely, $T(1)$ is defined as the weak-$\ast$ limit (in $\BMO$\,)
of $T(\phi_R)$.
}
Here,
 $\phi$ is a test function in  $C^\infty_c$ with $0 \leq \phi \leq 1$
such that  $\phi (x) =  1$ for $|x| \leq \frac{1}{2}$ and $\supp  \phi \subset B_0(1)$,
and
\begin{align}
\phi_R(x) = \phi(R^{-1} x).
\label{phi1}
\end{align}

\noi
Moreover,
proceeding as in \cite[Section 4]{BO2}
with \eqref{S1}, we have
\begin{align}
\|T(\phi_R)\|_{\BMO} \les 1,
\label{X5}
\end{align}

\noi
uniformly in $R > 0$;
see \cite[(4.6)]{BO2} with $b_1 = 1$.

In the following, we show that $T(\phi_R)$ indeed belongs to \CMO
for each $R \gg1 $ by verifying~\eqref{R4} in Lemma \ref{LEM:CMO}.
Fix  $M \in \mathbb{N}$.
Then, by imposing that
\begin{align}
\|\dd^\al \phi\|_{L^\infty} \leq 1
\label{bump1}
\end{align}

\noi
for all multi-indices $\al$ with $|\al| \leq M$,
the function $\phi$ defined above is a normalized bump function of order $M$.

Fix a  cube $Q = Q(\l, x_0)$ of side length $\l  > 0$ with center $x_0 \in\R^d$.
Set
\begin{align}
r = 6 \diam(Q) = 6 \sqrt d \l .
\label{X6}
\end{align}

\medskip

\noi
$\bullet$ {\bf Case 1:} $r\to \infty$.
\\
\indent
By Cauchy-Schwarz's inequality with \eqref{R0} and \eqref{X6}, we have
\begin{align*}
\ave_Q |T(\phi_R)|
= \frac 1{|Q|} \int_Q |T(\phi_R)(x) |dx
\le \frac 1{|Q|^\frac 12 } \|T(\phi_R)\|_{L^2}
\les \frac{R^\frac d 2}{r^\frac d2}
\too 0,
\end{align*}

\noi
as $r \to \infty$,
from which we obtain
\begin{align*}
\frac 1{|Q|}
\int_Q |T(\phi_R) - \ave_Q T(\phi_R)| dx
\le 2 \ave_Q |T(\phi_R)| \too 0,
\end{align*}

\noi
as $r \to \infty$.
In view of \eqref{R2}, this shows
\begin{align}
\G_2(T(\phi_R))
= \lim_{r \to \infty}\bigg(
\sup_{|Q| \ge r}
 \frac{1}{|Q|} \int_{Q} |T(\phi_R)(x) - \ave_{Q} T(\phi_R)|dx
\bigg)
= 0.
\label{X7}
\end{align}

\medskip

\noi
$\bullet$ {\bf Case 2:} $|x_0| \to \infty$.
\\
\indent
Fix $R,  r > 0$.
With  $\phi_Q = \phi^{x_0, r}$ as in \eqref{scaling},
write
 $T( \phi_R)$ as
\begin{align}
 T(\phi_R) = T( \phi_Q\phi_R) + T\big( (1-\phi_Q)\phi_R\big)
 =: \1 + \II.
\label{L8}
 \end{align}

\noi
By choosing $|x_0| \gg R$,
we have $\phi_Q\phi_R = 0$
and thus $\1 = 0$.
Hence, we only need to estimate the second term $\II$
on the right-hand side of  \eqref{L8}.
From the support condition
\begin{align*}
\supp (1- \phi_Q) \subset \R^d \setminus B_{x_0}( 3  \diam(Q))
\subset \R^d \setminus Q
\end{align*}

\noi
and \eqref{CZ1},
we have
\begin{align}
 \II(x) = 	
\int_{\R^d}  K(x, y) \big(1- \phi_Q(y) \big)
\phi_R(y)
dy
\label{Y1}
\end{align}

\noi
for any $x\in Q$.
By taking an average over $Q$, we then have
\begin{align}
\ave_Q \II
= \frac{1}{|Q|} \int_Q
\int_{\R^d}  K(z, y) \big(1- \phi_Q(y) \big) \phi_R(y) dydz.
\label{Y2}
\end{align}

\noi
By taking the difference of \eqref{Y1} and \eqref{Y2}, we have
\begin{align*}
\II(x) - \ave_Q \II
= \frac{1}{|Q|} \int_Q
\int_{\R^d} \big(K(x, y) -  K(z, y)\big) \big(1- \phi_Q(y) \big) \phi_R(y) dydz
%\label{Y3}
\end{align*}

\noi
for any $x \in Q$.
In particular,
from \eqref{CZ2}, we have
\begin{align}
\begin{split}
& \big|\II(x) - \ave_Q \II\big|\\
& \quad \les \frac{1}{|Q|} \int_Q
\intt_{|x - z| \leq \diam(Q) <  \frac 12 |x - y|}
\frac{|x-z|^\dl}{|x-y|^{d+\dl}} \big|1- \phi_Q(y) \big| |\phi_R(y)| dydz,
\end{split}
\label{Y4}
\end{align}

\noi
uniformly in  $x \in Q$.

By considering the support of $\phi_R$,
we have $|y| \les R$ in \eqref{Y4}.
Thus, for $|x_0| \gg R + r$,
it follows from $|x - x_0|\le r$ that
$|x- y| \sim |x| \sim |x_0|$.
Then, by performing the $y$-integration in~\eqref{Y4} on $\supp (\phi_R)$,
we obtain
\begin{align*}
\big|\II(x) - \ave_Q \II\big|
\les
\frac{r^\dl R^d}{|x_0|^{d+\dl}} \too 0,
\end{align*}

\noi
as $|x_0| \to \infty$,
uniformly in  $x \in Q$.
Hence, together with $\1 = 0$ for $|x_0| \gg R$, we obtain
\begin{align}
\frac 1{|Q|}
\int_Q |T(\phi_R) - \ave_Q T(\phi_R)| dx
 \too 0,
 \label{Y4a}
\end{align}

\noi
as $|x_0| \to \infty$.

Let  $\wt Q = \wt Q (\l,  x_1)$  be a cube
of side length $\l  > 0$ with center $ x_1 \in\R^d$.
Then, given $\wt x_0 \in \R^d$,
we have
\[\wt Q  + \wt x_0 = Q (\l, x_0), \]

\noi
where
 $ Q(\l, x_0)$ denotes
the cube of  side length $\l  > 0$ with center $x_0 := x_1+ \wt x_0$.
For fixed $x_1 \in \R^d$, we have $|x_0| \to \infty$ as $|\wt x_0|\to \infty$.
Therefore,
from
 \eqref{R3} and \eqref{Y4a}, we obtain
\begin{align}
\G_3^{\wt Q}(T(\phi_R)) =
\limsup_{\substack{\wt x_0 \in \R^d\\|\wt x_0| \to \infty}}\bigg(
 \frac{1}{|\wt Q|} \int_{\wt Q+\wt x_0} |T(\phi_R)(x) - \ave_{\wt Q+\wt x_0} T(\phi_R)|dx
\bigg)
= 0
\label{X8}
\end{align}

\noi
for any cube $\wt Q \subset \R^d$.

\medskip

\noi
$\bullet$ {\bf Case 3:} $r\to 0$.
\\
\indent
Fix $R > 0$ and $x_0 \in \R^d$.
In view of
\eqref{R0},
we have
\begin{align*}
\lim_{r \to 0} r^{-\frac d2} \|T(\phi^{x_0, r})\|_{L^2} = 0.
\end{align*}

\noi
Then, by setting $A_r = \min \big(r^{\frac 12} \|T(\phi^{x_0, r})\|_{L^2} ^{-\frac 1 d }, r^{-\frac 1{d+2}} \big)$,
we have
\begin{align}
A_r & \too \infty,  \label{Z0a}\\
A_r^\frac{d+2}{2} r & \too 0,
\label{Z0b}
\end{align}

\noi
and
\begin{align}
A_r^\frac d2 r^{-\frac d2} \|T(\phi^{x_0, r})\|_{L^2} \too 0
\quad \text{or equivalently}
\quad
r^{-\frac d2} \|T(\phi^{x_0, A_r r})\|_{L^2} \too 0,
\label{Z1}
\end{align}

\noi
as $r \to 0$, where, in \eqref{Z1}, we used the fact that $A_r r \to 0$ as $r\to 0$.
In the following, we assume $r> 0$ is sufficiently small
such that  $A_r > 1$.

We only consider the case $x_0 \in \supp \phi_R$,
since the case $x_0 \notin \supp \phi_R$ can be treated
in an analogous but easier manner.
Define $\wt \phi_Q $
by
\begin{align}
\wt \phi_Q = \phi^{x_0, A_r r}.
\label{Y5a}
\end{align}

\noi
Then, we can
write $T( \phi_R)$ as
\begin{align}
\begin{split}
 T(\phi_R)
 & =
\phi_R(x_0) T(\wt  \phi_Q) +
 T\big(\wt  \phi_Q(-\phi_R(x_0)+\phi_R) \big)+ T\big( (1-\wt \phi_Q)\phi_R\big)\\
&  =: \wt \1 + \wt \II+ \wt \III.
\end{split}
\label{Y5}
 \end{align}

From
\eqref{Y5a} and \eqref{Z1}, we have
\begin{align}
\ave_Q|\wt  \1|  \le \frac{|\phi_R(x_0)|}{|Q|} \int_{Q} |T( \wt \phi_Q)| dx
\les   \frac 1{ r^{\frac d 2}} \big\|T\big(  \phi^{x_0, A_r r}\big)\big\|_{L^2}
\too 0,
\label{X9}
\end{align}

\noi
as $r\to 0$.

By the mean value theorem with \eqref{phi1}, \eqref{bump1}, and \eqref{Y5a}
(which implies $|x-x_0|\les A_r r$ on $\supp \wt \phi_Q$), we have
\begin{align}
|\wt \phi_Q(x) (-\phi_R(x_0) +\phi_R(x))|
\les R^{-1} |\wt \phi_Q(x)| |x-x_0|
\les R^{-1} A_r r |\wt \phi_Q(x)|.
\label{bump2}
\end{align}

\noi
Then, from
the boundedness of $T$ on $L^2(\R^d)$,
\eqref{bump2},
\eqref{Y5a}, and \eqref{Z0b}, we have
\begin{align}
\ave_Q|\wt  \II|
\le \frac{1}{|Q|^\frac 12} \| \wt \II\|_{L^2(Q)}
\les   \frac {A_r r } { R r^{\frac d 2}}
\|\wt \phi_Q\|_{L^2(\R^d)}
\sim    \frac {A_r^{\frac {d+2}2} r } { R}
\too 0,
\label{X10}
\end{align}

\noi
as $r\to 0$.

As for $\wt \III$,
proceeding as in Case 2, we have
\begin{align*}
\wt \III(x) - \ave_Q \wt \III
= \frac{1}{|Q|} \int_Q
\int_{\R^d} \big(K(x, y) -  K(z, y)\big) \big(1- \wt \phi_Q(y) \big) \phi_R(y) dydz
\end{align*}

\noi
for any $x \in Q$.
Then,
from \eqref{CZ2} with $A_r > 1$, we have
\begin{align*}
& \big|\wt \III(x) - \ave_Q \wt \III\big|\\
& \quad \les \frac{1}{|Q|} \int_Q
\intt_{|x - z| \leq \diam(Q) < \frac 1{2A_r} |x - y|}
\frac{|x-z|^\dl}{|x-y|^{d+\dl}} \big|1- \phi_Q(y) \big| |\phi_R(y)| dydz,
\end{align*}

\noi
uniformly in  $x \in Q$.
By integrating in $y$ and $z$ in the polar coordinates (centered at $x$
with $r_1 = |x- y| \ges A_r r $ and $r_2 = |x - z| \les r$)
with $|Q| \sim r^d$
and \eqref{Z0a},
 we have
\begin{align}
\begin{split}
 \big|\wt \III(x) - \ave_Q \wt \III\big|
 \les \frac{1}{r^d}
\int_0^{c_2 r}
\int_{c_1 A_r r}^\infty \frac{1}{r_1^{1+\dl}} dr_1
r_2^{d - 1 + \dl }
dr_2
\sim \frac{1}{A_r^\dl}\too 0,
\end{split}
\label{X11}
\end{align}

\noi
as $r\to 0$,
uniformly in  $x \in Q$.

Hence,  putting \eqref{Y5}, \eqref{X9},
\eqref{X10}, and \eqref{X11} together,
\begin{align}
\frac 1{|Q|}
\int_Q |T(\phi_R) - \ave_Q T(\phi_R)| dx
 \too 0,
 \label{X11a}
\end{align}

\noi
as $r \to0$,
 uniformly in $R \gg 1$.
From the definition \eqref{R1} of $\G_1$, this shows
\begin{align}
\G_1(T(\phi_R))
=
\lim_{r \to 0}\bigg(
\sup_{\substack{|Q|  \le r}}
 \frac{1}{|Q|} \int_{Q} |T(\phi_R)(x) - \ave_{Q} T(\phi_R)|dx
\bigg)
= 0.
\label{X12}
\end{align}

Therefore,
from Lemma \ref{LEM:CMO} with
\eqref{X7},
\eqref{X8},
and
\eqref{X12},
we conclude that
\begin{align}
T(\phi_R) \in \CMO
\label{X13}
\end{align}

\noi
for each $R>0$.

Let $\{R_j\}_{j \in \N}$ be an increasing sequence
of positive numbers such that $R_j \to \infty$ as $j \to \infty$.
Then, the sequence
$\{\phi_{R_j}\}_{j \in \N}$
converges to the constant function $\phi(0) = 1$
in the weak-$\ast$ topology of $L^\infty$.
Moreover,
from~\eqref{X5} and \eqref{X13},
the sequence
$\{T(\phi_{R_j})\}_{j \in \N}$  is bounded in $\CMO$\,.
Therefore,
from the $L^\infty$-condition~(ii),
we conclude that
there exists a subsequence
$\{T(\phi_{R_{j_k}})\}_{k \in \N}$
converging strongly to $T(1)$ in $\CMO$\,,
which in particular implies
$T(1) \in \CMO$\,.

We have now verified (B.i) and (B.ii)
in Theorem \ref{THM:MS},
which in turn implies that $T$ is compact on $L^2$.
This concludes the proof of Theorem \ref{THM:1}.

\begin{remark}\label{REM:VMO}
\rm
In view of  \eqref{X11a},
the convergence in \eqref{X12}
is uniform in $R \gg 1$.
However,
the convergence in \eqref{X7}
and
\eqref{X8} in Cases 1 and 2, respectively,
is not uniform in $R \gg 1$.
This is the reason that we needed to introduce
the $L^\infty$-condition (ii) in Theorem \ref{THM:1}.
\end{remark}

\smallskip

\begin{remark}\label{REM:VMOii}
\rm
In view of $(\CMO\,)^* = H^1$
(see  \cite[Theorem (4.1) on p.\,638]{CW}\footnote{We note that the space
$\textit{VMO}$ in \cite{CW} corresponds to $\CMO$ in the current work.})
and that $T(\phi_{R}) \in \CMO$\,, $R>0$,
it follows from \eqref{X4}
(which holds for any $g \in H^1$)
that $T(1)$ is a weak limit of $\CMO$ functions
with respect to the weak topology on $\CMO$\,.
Hence, if $\CMO$ {\it were} weakly sequentially complete,
we would be able to conclude $T(1) \in \CMO$\, without
the $L^\infty$-condition~(ii) in Theorem~\ref{THM:1}.
However, this is not the case.
See Appendix \ref{SEC:weak}.

\end{remark}

\section{Mapping properties of Calder\'on-Zygmund operators into \CMO}
\label{SEC:5}

While every  Calder\'on-Zygmund operator maps $L^\infty$ to $\BMO$,
 compact ones in particular map
$L^\infty$ to $\CMO$
as pointed out in Remark \ref{REM:CMO}\,(b).
 We will show that, in general, without the compactness assumption, \CZ operators fail to have this property.
Moreover, this failure persists even if we restrict input functions to
smaller  subspaces of $L^\infty$,
consisting of smooth functions; see Proposition~\ref{no}.

In the case of \CZ operators of convolution type,
we do have that they map $C_0$ into $\CMO\,$. This claim appears in \cite[p.\,180]{Stein}\footnote{The space $\VMO$ in \cite{Stein}
also corresponds to $\CMO$ in our discussion.} without a proof.
In  Proposition~\ref{yes},
we prove this  claim in a little more generality.
We note that, in
 \cite[p.\,207]{G2} and \cite[p.\,10]{MitSto}, analogous claims are made
for general \CZ operators (in particular for those of non-convolution type),
which are unfortunately  incorrect.
In   Proposition~\ref{no},
we construct an example
of a \CZ operator $T$ such that  $T(C^\infty_c) \not\subset \CMO$\,.

We first recall some further  properties of \CZ operators.
Given  a \CZ operator $T$, we have
$T(f) \in L^1$
for  $f \in (C^\infty_c)_0$.
Here,   $(C^\infty_c)_0$
denotes
the subspace of $C^\infty_c$  consisting of mean-zero functions,
which is  dense in  $H^1$.
If, in addition, $T$ maps $H^1$ into itself,
 then
 we have
\[
0= \int_{\R^d} T(f)(x)\, dx  =
\jb{1, T(f) } = \jb{ T^*(1), f}
\]

\noi
for any  $f \in (C^\infty_c)_0$.
Namely, we have
$T^*(1)=0$.
In fact, such a condition is sufficient
for  boundedness of $T$ on  $H^1$.

\begin{lemma}\label{h1-to-h1}
Let $T$ be a \CZ operator of convolution type\footnote{Namely,  the kernel $K$ of the operator $T$ is of the form $K(x,y)=k(x-y)$ such that  $Tf=k * f$, at least in the distributional sense.
Note that, for a \CZ operator $T$ of convolution type,  we have $T(1) = T^*(1) = 0$ (as elements in $\BMO$\,, namely, modulo additive constants).} or, more generally,
let $T$
be  a \CZ operator with  $T^*(1)=0$.
Then,  $T$ maps $ H^1$ into itself.
\end{lemma}

In the convolution case,  this result goes back to Fefferman and Stein \cite[Corollary~1]{FS}.
For non-convolution operators with  $T^*(1)=0$, this result was proved by Alvarez and Milman \cite[Theorem~1.1]{AM}.
Under the additional assumption that  $T(1)=0$,
this result
%Lemma \ref{h1-to-h1}
was shown by
 Frazier et al. \cite{FHJW,  FTW}.
  This case subsumes the Fefferman-Stein result
  on \CZ operators of convolution type; see, for example, the comments in \cite[p.\,67]{FTW}.
  Note that if $T(1)=T^*(1)=0$, then
  we have  both $T,\, T^*: H^1 \to H^1$ and $\BMO \to \BMO$\,.
The next two results  show that
the cancellation condition $T(1)=T^*(1)=0$ also provides  additional mapping properties.

\begin{lemma}\label{smooth-to-c0}
Let $T$ be a \CZ operator of convolution type or, more generally,
let $T$
be  a \CZ operator with   $T(1)=T^*(1)=0$.
Then,  for any $f \in C^\infty_c$,  we have
$T(f) \in C_0$.
\end{lemma}

\begin{proof}
We note that this is very easy to prove
 in the convolution case.
Since $T$ is bounded on $L^2$, there exists  $m\in L^\infty$ such that $\widehat {T(f)} = m \ft  f$.
Since $f \in C^\infty_c$, we have $\widehat {f} \in \S$ and thus  $m \ft f\in L^1$.
Then, by the Riemann-Lebesgue lemma, we obtain
$T(f) \in C_0$.

In the more general case, we cannot take such direct advantage of the Fourier transform.
Using certain molecular decompositions,
Meyer \cite[Proposition 2]{M} showed that, under the hypothesis of the lemma, $T$ maps H\"older continuous functions into H\"older continuous functions.
In particular, for $f \in C^\infty_c$,   $ T(f)$ is continuous.

Suppose that $\supp f \subset B_0(R) \subset \R^d$.
Then, if $|x| \ge 2R$,  we have $|x- y| \sim |x|$ for any $y \in \supp f$.
Hence, from \eqref{CZ1a}, we have
\begin{equation}\label{vanishes}
| T(f)(x)| = \left| \int_{\supp f} K(x,y) f(y)\,dy\right| \lesssim   \int_{\supp f} \frac{1} {|x|^d} |f(y)|\,dy \lesssim \frac{\| f\|_{L^1}} {|x|^d} \too 0
\end{equation}

\noi
as $|x| \to \infty$.
Therefore, we conclude that
$T(f) \in C_0$.
\end{proof}

As stated
in
Remark~\ref{REM:CMO}\,(e) and (f),
under the cancellation condition $T(1)=T^*(1)=0$,
a compact \CZ operator $T$ is compact
from \BMO to \CMO
and
also from $H^1$ into itself,
where the latter claim follows
from the former via duality.
In fact,
a compact \CZ operator $T$ is compact  from $\CMO$ to itself
and thus  from Lemma~\ref{LEM:com}\,(ii),
we see that
$T$ is compact from $H^1$ into itself.
Nonetheless, %in the next proposition,
using the characterization~\eqref{H1}
of the Hardy space $H^1$,
we present a direct proof of
the compactness on $H^1$
of a compact \CZ operator $T$
under the cancellation condition $T(1)=T^*(1)=0$
without using such a duality argument.

\begin{proposition}\label{PROP:cpt}
Let $T$
be  a \CZ operator
   of convolution type or, more generally,
let $T$
be
  a \CZ operator with $T(1) = T^*(1) = 0$.
  Suppose that $T$ is compact on $L^2$.
  Then, $T$ is compact from  $H^1$ to  itself.
\end{proposition}

\begin{proof}

Let $T$ be a compact \CZ operator with $T(1) = T^*(1) = 0$.
By Remark~\ref{REM:CMO}\,(c),
 we know that $T :H^1 \to L^1$ is compact.
On the other hand,
recalling from
 \cite[Section 8.9]{MC}
 that the collection of \CZ operators $S$
 satisfying the cancellation condition
 $S(1) = S^*(1) = 0$
 forms an algebra,
 we see that
  $R_jT$ is also a Calder\'on-Zygmund operator, $j = 1, \dots, d$
  (since $R_j(1) = R_j^*(1) = 0$ as a convolution operator).
 Moreover, since $T$ is compact on $L^2$ and $R_j$ is bounded on $L^2$,
 the composition $R_jT$  is compact  on $L^2$.
 Therefore,
 the desired compactness on $H^1$ now  follows  from the definition \eqref{H1} of the $H^1$-norm:
\[
\|Tf||_{H^1} = \|Tf\|_{L^1} + \sum_{j=1}^{d} \| (R_jT)f\|_{L^1},
\]

\noi
since each operator appearing on the right-hand side is compact from $H^1$ into $L^1$.
\end{proof}

We now state the first main result of this section
which extends
the aforementioned claim in \cite[p.\,180]{Stein}
to a slightly general setting.

\begin{proposition} \label{yes}
Let $T$ be a \CZ operator of convolution type or, more generally,
let $T$
be  a \CZ operator with   $T(1)=T^*(1)=0$.
Then,  $T$ maps $C_0$ to $\CMO$\,.
\end{proposition}

\begin{proof}
Let $f \in C_0$.
Then, there exists  $\{f_n\}_{n\in \N} \subset C^\infty_c$ such that
$f_n$ converges to  $f$ in $L^\infty$ (and thus in $\BMO$\,) as $n \to \infty$.
 By the boundedness of $T:L^\infty \to BMO$, we  have $T(f), \, T(f_n) \in BMO$
 for any $n \in \N$.
Then, from the definition of the transpose and   Lemma~\ref{h1-to-h1}, we have
\begin{align}
\begin{split}
|\jb{T(f)-T(f_n), g}|
& = |\jb{ f-f_n, T^*(g)}| \leq \| f-f_n\|_{\BMO} \|T^*(g)\|_{H^1} \\
& \lesssim \| f-f_n\|_{\BMO} \|g\|_{H^1}
\end{split}
\label{X1}
\end{align}

\noi
for any $g\in H^1$.
Recalling %Since $H^1$ is {\it norming} in $\BMO$\,, that is
$$
\|b\|_{\BMO} = \sup_{\|g\|_{H^1}=1} |\langle b, g\rangle|,
$$

\noi
we see from \eqref{X1} that
 $T(f_n)$ converges to $T(f)$ in $\BMO$\,
 as $n \to \infty$.
 On the other hand,
 it follows from
 Lemma~\ref{smooth-to-c0} that  $T(f_n) \in C_0$ for each $n \in \N$.
Therefore,
recalling from
 \cite[Th\'eor\`eme~7]{Bour}
that the closure of $C_0$ in \BMO is  $\CMO$,
we conclude that
$ T(f) \in \CMO$.
\end{proof}

Our next goal is to show that
Proposition \ref{yes} does not hold in general
without the cancellation condition
   $T(1)=T^*(1)=0$;
see Proposition \ref{no}.
For this purpose,
 we first need to recall the notion of a paraproduct.

Let $\varphi, \psi\in C_c^\infty$ be radial functions such that
$\supp (\varphi)\subset B_0(1)$,
$\widehat\psi (0)=0$, and
\begin{align*}
\int_0^\infty |\ft \psi(te_1)|^2\frac{dt}{t}=1,
\end{align*}

\noi	
 where $e_1=(1, 0,..., 0)\in\R^d$. For $t\in\R_{+}$,
we also define  the linear convolution operators $P_t$ and $Q_t$
by
\begin{align}
P_t f=\varphi_t *f
\qquad \text{and} \qquad Q_t f=\psi_t*f,
\label{Cal1}
\end{align}

\noi
where $h_t=t^{-d}h(t^{-1}\cdot)$
for a function $h$ on $\R^d$.
Then,  the Calder\'on reproducing formula \cite{Cal}
states that
\begin{align*}
\int_0^\infty Q_t^2 f \frac{dt}{t}=f
\end{align*}

\noi
holds in $L^2$,
where  $Q_t^2 f=Q_t(Q_t f)=\psi_t*\psi_t*f$.
Given $b\in\BMO$,
we  define
a paraproduct $\Pi_b$ by
\begin{align}
\Pi_b (f) =\int_0^\infty Q_t\big((Q_t b)(P_t f)\big) \frac{dt}{t},
\label{P0}
\end{align}

\noi
a priori defined for $f \in \S$.
It is well known by now that if $b\in \BMO$ then $\Pi_b$ is  a \CZ operator,  satisfying
\begin{align}
\Pi_b(1) = b \qquad \text{and}\qquad
\Pi^*_b(1) = 0,
\label{P1}
\end{align}
\noi
where once again the equalities are understood by testing
against   functions in $(C^\infty_c)_0$;
see, for example, \cite[Lemmas 1 and 2]{DJ}.
See also \eqref{X4}.

In what follows, we will prove that, for an appropriate symbol $b \in \BMO\setminus \CMO$\,,
the paraproduct $\Pi_b$ does {\it not} map $C_c^\infty$ into $\CMO$\,.
A naive approach would be to proceed as follows. Fix $b \in \BMO \setminus \CMO$\,.
Assuming that $\Pi_b$ maps $C_c^\infty$ into $\CMO$, we would have  $\Pi_b(\phi_R) \in \CMO$, where $\phi_R$ is as in \eqref{phi1}. If \CMO {\it were} weakly sequentially complete, we would
then have  $b=\Pi_b(1)\in \CMO$,
since, as seen in \eqref{X4},
$\Pi_b(1)$ is defined as the limit of $\Pi_b(\phi_R)$
by testing against $H^1$-functions,
namely as a limit with respect to the weak topology in $\CMO$\,,
in view of the duality $(\CMO)^* = H^1$.
This would be a contradiction with the fact that $b \in \BMO\setminus \CMO$\,.
However, this naive approach breaks down precisely because \CMO is \emph{not} sequentially complete; see again Appendix \ref{SEC:weak} and Remark \ref{REM:VMOii}. Therefore, a more delicate argument and a more careful choice of the symbol $b$ is needed here.

\begin{proposition} \label{no}
There exists a \CZ operator, in particular a paraproduct $\Pi_b$
with $b \in \BMO \setminus \CMO$ such that $\Pi_b(C^\infty_c) \notin \CMO$\,.

\end{proposition}

We begin by showing that, for our argument,  the \CMO space can be ``replaced''  by  $\VMO$, the space of functions of \emph{vanishing mean oscillation}
which
is defined as the closure
of $C^\infty\cap \BMO\,$ in the $\BMO$ topology.
In particular, we have
\begin{align}
 \CMO \subset \VMO \subset \BMO.
 \label{embed1}
\end{align}

\noi
We recall
 the following characterization of
$\VMO$\,; see, for example,
\cite[Th\'eor\`eme 5]{Bour}.

\begin{lemma}\label{LEM:VMO}
Let $f \in \BMO$\,.
Then, $f \in \VMO$
if and only if
$\G_1(f)  = 0$,
where $\G_1$ is as in~\eqref{R1}.
\end{lemma}

The next lemma shows that
\CMO is the subspace of function in \VMO which ``vanish" at infinity;
see \cite[Th\'eor\`eme 7]{Bour}.

\begin{lemma}\label{vanishatinfty}
Let $\BMO_c$ denote the subspace
of $\BMO$ consisting of functions with compact support.
Then, we have
\[\CMO = \VMO \cap \cj{\BMO_c},\]

\noi
where  the closure is with respect to the $\BMO$\,-norm.
\end{lemma}

 Let $f \in C^\infty_c(\R^d)$.
 As we saw in \eqref{vanishes},  we have that $T(f)$ (which is in $\BMO$\,) vanishes at infinity.
 Hence, from \cite[Th\'eor\`eme 6]{Bour},
we conclude that $T(f) \in \cj{\BMO_c}$.
Therefore, it follows from  Lemma~\ref{vanishatinfty}
that $T(f) \in \CMO$ if and only if
$T(f) \in \VMO$\,.

We are now ready to prove Proposition \ref{no}.

\begin{proof}[Proof of Proposition \ref{no}]
In the following, we work on $\R$.
For a suitable choice of $b \in \BMO \setminus \VMO$,
we show that $\Pi_b(\phi_R) \notin \VMO$ for any finite $R \gg 1$
(and hence $\Pi_b(\phi_R) \notin \CMO$ in view of
\eqref{embed1}),
%the observation after Lemma \ref{embed1}),
where $\phi_R$ is as in \eqref{phi1}.

From \eqref{Cal1}, we have
\begin{align}
Q_t\big((Q_t b)(P_t \phi_R)\big)(x)
& =\frac{1}{t^{2d}}
\int_{\R^2} \psi\Big(\frac{ x - y}{t}\Big) Q_t b(y)
\varphi\Big(\frac{ y - z}{t}\Big)
\phi_R(z) dz dy.
\label{P5}
\end{align}

\noi
Suppose that $\supp \psi$, $\supp \varphi \subset [-C_0, C_0]$
for some $C_0 > 0$.
Then, we have
$ |x - y|, |y - z| \le t C_0$
on the right-hand side of \eqref{P5}.
Thus, by the triangle inequality, we have
\begin{align}
|z| \le 2tC_0 + |x| \le  2C_0 + 1
\label{P6}
\end{align}

\noi
for any $0 < t \le 1$ and $|x| \le 1$.
Hence,  from \eqref{P5} with \eqref{phi1} and \eqref{P6}, we have
\begin{align}
\begin{split}
Q_t\big((Q_t b)(P_t \phi_R)\big)(x)
& = Q_t\big((Q_t b)(P_t 1)\big)(x)
\end{split}
\label{P7}
\end{align}

\noi
for any $0 < t \le 1$,  $|x| \le 1$,
and $R \ge 4C_0 + 2$.

We now choose a suitable function
$b \in \BMO \setminus \! \VMO$\,.
Define an odd function  $b \in
  C^\infty((-\infty, 0) \cup (0, \infty); \R)$ by setting
\begin{align*}
b(x) =
\begin{cases}
1, & \text{for $0 <  x \le 1$},\\
-1, & \text{for $-1 \le x <0$},\\
0, & \text{for $x = 0 $ and $ |x| \ge 2$}
\end{cases}
%\label{Cal2}
\end{align*}

\noi
and smoothly interpolating  on the intervals
$[-2, -1]$ and
$[1, 2]$ in such a way that $b$ is odd.
While we have $b \in L^\infty \subset \BMO$\,,
we see that $b \notin \VMO$
due to the jump at $x = 0$.
Indeed, with $Q = [-a, a]$, we have
\begin{align*}
 M_Q(b)=  \frac{1}{|Q|} \int_Q|b(x) - \ave_Q b |\,dx
=
 \frac{1}{|Q|} \int_Q|b(x)  |\, dx
= 1
\end{align*}

\noi
for  any $0 < a < 1$,
which implies $b \notin \VMO$ in view of Lemma \ref{LEM:VMO}.

From \eqref{P0} and \eqref{P1}, we have
\begin{align*}
b = \Pi_b (1) =\int_0^\infty Q_t\big((Q_t b)(P_t 1)\big) \frac{dt}{t},
\end{align*}

\noi
where the second inequality
holds as elements in the dual space of $H^1$
(namely, as elements in $\BMO$\,).
Write $b = b_1 + b_2$,
where $b_1$ is defined by
\begin{align}
b_1 = \int_0^1 Q_t\big((Q_t b)(P_t 1)\big) \frac{dt}{t}.
\label{P9}
\end{align}

\noi
It is easy to see that $b_2$ is a smooth function on $\R$
since it essentially corresponds to the low frequency part
of the paraproduct $\Pi_b(1)$.
Moreover,  since $\varphi$ and $\psi$
are even  and $b$ is odd,
it follows from \eqref{Cal1} that  $b_2$ is a smooth odd function
which in particular implies $b_2(0) = 0$.
Hence,
we conclude that $b_1$ has a jump
discontinuity at $x = 0$:
\begin{align}
b_1(0+) - b_1(0-) = 2.
\label{P10}
\end{align}

\noi

Fix $R \ge 4C_0 + 2$.
Write $\Pi_b(\phi_R) = A_1 + A_2$,
where $A_1$ is defined by
\begin{align}
A_1 = \int_0^1 Q_t\big((Q_t b)(P_t \phi_R)\big) \frac{dt}{t}.
\label{P11}
\end{align}

\noi
Then, from \eqref{P7}, \eqref{P9},
and \eqref{P11}, we have
$A_1(x) = b_1(x)$
for any $|x| \le 1$.
Hence, from~\eqref{P10}, we obtain
\begin{align*}
A_1(0+) - A_1(0-) = 2,
\end{align*}

\noi
where we used the fact that $A_2(0) = 0$ just as for $b_2$
discussed above.
By noting that  $A_2$ is a smooth function on $\R$,
we then have
\begin{align*}
\Pi_b(\phi_R)(0+) - \Pi_b(\phi_R)(0-) = 2.
\end{align*}

\noi
Hence,  by computing $M_Q(\Pi_b(\phi_R))$
for $Q = (-a, a)$ with $a$ tending to $0$,
we see that
 $\G_1(\Pi_b(\phi_R)) \ne 0$,
where
$M_Q$ and
$\G_1$ are as in
\eqref{M1} and
\eqref{R1}, respectively.
Therefore, we conclude from Lemma \ref{LEM:VMO}
that $\Pi_b(\phi_R)\notin \VMO$\,.
This shows that, in general, a \CZ operator
does {\it not} map $C_c^\infty$ into $\VMO$ or $\CMO$.
\end{proof}

\begin{remark}\rm

As subspaces of $\BMO$, the difference between \CMO and \VMO
 appears only in the global aspect of their elements,
namely the decay at spatial infinity of their elements
(just like the difference between $C_0$ and $C$ as subspaces of $L^\infty$).
In the counterexample presented above,
the failure of $C^\infty_c$-to-$\CMO$ mapping property comes from
the local   property of elements
shared by \CMO and $\VMO$\,, namely, they can not have an isolated jump discontinuity,
while  an element in $\BMO$ can have an isolated jump discontinuity,

\end{remark}

\appendix

\section{Lack of compactness of convolution operators in Lebesgue spaces}

From the result in the previous section,
 it may be  tempting to think that some  \CZ operators of convolution type may be compact. However it is known that that cannot be true. See for example, Feitchinger \cite[p.\,310]{F} where it is stated  ``that in the case of a non-compact group there do not
exist compact multipliers between two isometrically translation invariant
spaces". We will provide a simple proof of this fact in the context of
the Lebesgue spaces $L^p(\R^d)$.

Consider first the following simple but illustrative example.
Given $\eps>0$, let  $J^{-\eps}$ be the  Bessel potential of order $\eps > 0$.
 Namely, $J^{-\eps}$ is a Fourier multiplier operator with a bounded multiplier $(1+|\xi|^2)^{-\frac \eps2}$.
  It is not difficult to show that $J^{-\eps}$ is a Calder\'on-Zygmund operator.
If $J^{-\eps}$ were  compact on $L^2(\R^d)$,
then this would imply that
 $L^2(\mathbb R^d)$ compactly embeds into  $H^{-\eps}(\mathbb R^d)$, which is false.
 However,  the situation is different in the compact setting.
 We point out that if the underlying space is the $d$-dimensional torus $\T^d = (\R/\Z)^d$,
then Rellich's lemma states that the embedding  $L^2(\T^d)  \hookrightarrow H^{-\eps}(\T^d)$
is indeed compact.
Namely, the lack of compactness of the Bessel potential $J^{-\eps}$
on $L^2(\R^d)$ comes from the fact that functions on $\R^d$
can run off  to the spatial infinity.

The example above is
representative
%illustrative
of the following general fact.

\begin{proposition}
\label{multiplier-noncompact}
Let $1\leq p < \infty.$
A non-trivial\footnote{Namely, $T\ne 0$.} translation invariant  operator $T$
%that is bounded on $L^p(\R^d)$
is not a compact operator
on $L^p(\R^d)$.

\end{proposition}

\begin{proof}
Without loss of generality, we assume that $T$
is  bounded on $L^p(\R^d)$ since, otherwise, $T$ is obviously not  compact.
Since  $T$ is translation invariant, we have
\begin{align}
T(\tau_yf)= \tau_yT(f),
\label{Cal4}
\end{align}

\noi
where $\tau_yf(x)= f(x-y)$.

Fix $f \in L^p$ such that $T(f) \ne 0$.
Given a sequence  $\{y_n\}_{n\in \N}\subset \R^d$ with $|y_n|\to \infty$,
let
$f_n=\tau_{y_n}f$.
Clearly, the sequence $\{f_n\}_{n \in \N}$ is bounded in $L^p$.
If $T$ were compact on $L^p$, then
it would follow from
the Kolmogorov-Riesz theorem
(see \cite[Theorem on p.\,275]{Y}; see also \cite[Theorem 5]{HHM})
that there exists
a subsequence  $\{f_{n_k} \}_{k \in \N}$
such that,
given any $\eps > 0$,
there exists $R = R(\eps) > 0$ such that
\begin{align}
\int_{|x| > R} |T(f_{n_k})(x)|^p dx < \eps^p
\label{Cal3}
\end{align}

\noi
for any $k \in \N$.

Fix $0 < \eps  < \frac 12 \| Tf\|_{L^p} $.
Suppose that there exists $R = R(\eps) > 0$
such that \eqref{Cal3} holds.
Then, from the definition of $f_{n_k}$, \eqref{Cal4},
a change of variables, and the dominated convergence theorem
together with the $L^p$-boundedness of $T$,
we have
\begin{align*}
 \lim_{k \to \infty} \int_{|x| > R} |T(f_{n_k})(x)|^p dx
&  =
\lim_{k \to \infty}\int_{|x| > R} |T(f)(x- y_{n_k})|^p dx\\
& =
\lim_{k \to \infty}\int_{|x + y_{n_k}| > R} |T(f)(x)|^p dx
= \| Tf\|_{L^p}^p > 2^p \eps^p,
\end{align*}

\noi
which is a contradiction to \eqref{Cal3}.
Therefore, we conclude that  a
non-trivial translation invariant
 $T$ is not compact on $L^p$.
\end{proof}

\section{{\it BMO}, {\it VMO}, and {\it CMO} are not weakly sequentially complete}
\label{SEC:weak}

In this appendix,  we study
weakly sequential completeness properties
for the $\BMO$\,-type spaces
which is closely related to the need for the $L^\infty$-condition in the proof of Theorem \ref{THM:1};
see
Remark \ref{REM:VMO}\,(ii).

Weakly sequential completeness of the Hardy space $H^1$
follows as an immediate  consequence of
\cite[Lemma~(4.2)]{CW}. A natural question is then to ask whether the dual and pre-dual of $H^1$ (\BMO and
$\CMO$\,, respectively) are also weakly sequential complete. We will show that the answer is negative.

\begin{proposition}\label{PROP:D3}
$\BMO$, $\VMO$, and $\CMO$
are not weakly sequentially complete.
\end{proposition}

Before proceeding to a proof of Proposition \ref{PROP:D3},
we state several lemmas.
We first recall  the following lemma from
\cite[Proposition 2.3]{BL}.

\begin{lemma}\label{LEM:B3}
Let $X$ be a weakly sequentially complete Banach space.
Then, any closed subspace  $Y$ of $X$  is a weakly sequentially complete Banach space.
\end{lemma}

Let $c_0=c_0 (\N) \subset \l^\infty(\N)$
denote the  space of
bounded sequences
${\bf a} = \{a_n\}_{n \in \N}$
such that $\displaystyle\lim_{n \to \infty}a_n = 0$. If we endow $c_0$ with the
supremum norm, then  we have $(c_0)^* =  \l^1(\N)= \l^1$,
where the duality pairing is given by
\[\jb{ {\bf a}, {\bf b}} = \sum_{n=1}^\infty a_nb_n\]

\noi
for ${\bf a}\in c_0$ and
${\bf b} \in  \l^1$.
The space $c_0$ provides a useful example for many functional analytic situations related to weak sequential (non-)completeness. For the readers convenience, we include a proof of the following fact.

\begin{lemma}\label{LEM:D1}
$c_0$ is not weakly sequentially complete.

\end{lemma}

\begin{proof}
Given $k \in \N$,
define a sequence ${\bf a}_k = \{a_{kn}\}_{n \in \N}$
by setting
$a_{kn} = 1$ for $n \le k$ and $a_{kn} = 0$ for $n > k$.
Then, it is easy to verify that ${\bf a}_k \in c_0$ for each $k \in \N$
and that ${\bf a}_k$ is weakly Cauchy.
However, its pointwise limit
is a constant sequence with value 1 which is not in $c_0$.
This proves the claim.
\end{proof}

Given a Banach space $X$,
we say that
$X$ contains a copy of $c_0$,
if
there exist a sequence $\{x_n\}_{n \in \N}$
of unit vectors and constants $C_1, C_2 > 0$ such that
\begin{align}
C_1 \sup_{n \in \N} |a_n|
\le \bigg\|\sum_{n = 1}^\infty
a_n x_n \bigg\|_X
\le C_2 \sup_{n \in \N} |a_n|
\label{D1}
\end{align}

\noi
for any ${\bf a} = \{a_n\}_{n \in \N} \in c_0$.

Let
\begin{align}
 Y = {\bigg\{
x = \sum_{n = 1}^\infty a_n x_n : \{a_n\}_{n \in \N} \in c_0
\bigg\}}.
\label{D1a}
\end{align}

\noi
In view of \eqref{D1},
we see that  every element in $Y$ is uniquely represented as
\[{\bf a} \cdot {\bf x} = \sum_{n = 1}^\infty a_n x_n \]

 \noi
 for some ${\bf a} = \{a_n\}_{n\in \N} \in c_0$.
  Moreover, $Y$ is isomorphic and ``almost isometric" to $c_0$, meaning that
$$
\|{\bf a}\|_{c_0} \approx \| {\bf a} \cdot {\bf x}\|_{X}
$$
in the sense of \eqref{D1}.
Once again,
in view of \eqref{D1},
we see  that $Y$ is closed in $X$ because $c_0$ is a complete Banach space.

\begin{lemma}\label{LEM:D2}
Let $X$ be a Banach space.
If $X$ contains a copy of $c_0$,
then $X$ is not weakly sequentially complete.

\end{lemma}

Under an additional assumption that the space $X$ has
an unconditional basis,
the converse of Lemma \ref{LEM:D2} also holds true;
see \cite[Theorem 1.c.10]{LT}.

\begin{proof}[Proof of Lemma \ref{LEM:D2}]
Since $X$ contains a copy of $c_0$,
there exists
$\{x_n\}_{n \in \N} \subset X$,
satisfying~\eqref{D1}.
Define a subspace $Y$ of $X$ as in \eqref{D1a}.
Recalling
that $c_0^* = \l^1$,
we see that $Y^* \cong \l^1$.
Indeed, given ${\bf b} = \{b_n\}_{n \in \N} \in \l^1$,
we set  ${\bf b}(x) = \sum_{n = 1}^\infty a_n b_n$
for $x = {\bf a} \cdot {\bf x} \in Y$.
By noting that the $Y$-$Y^*$ duality pairing agrees with the $c_0$-$\l^1$
duality pairing,
it follows   from
 Lemma \ref{LEM:D1}
 that
$Y$ is not  weakly sequentially complete.
Finally, by recalling  that $Y$ is a closed subspace of $X$,
we conclude from Lemma \ref{LEM:B3}
that $X$ is
not weakly sequentially complete.
\end{proof}

We are now ready to present a proof of
Proposition \ref{PROP:D3}.

\begin{proof}[Proof of Proposition \ref{PROP:D3}]

We only consider the $d = 1$ case.
Define a function $\varphi$ on $\R$ by setting
\begin{align*}
\varphi (x) = \begin{cases}
1, & 0 \le x\le 1, \\
-1, & -1 \le x < 0, \\
0, & \text{otherwise}.
\end{cases}
\end{align*}
	
\noi
Then, a direct computation shows that $\|f\|_{\BMO} = 1$.
While the function $f$ suffices for showing
weakly sequential non-completeness of $\BMO$\,,
we need to smooth it out to treat the $\VMO$ and $\CMO$
cases since $\varphi \notin \VMO$\,.

Let $\rho:\R \to [0, 1]$
be a smooth even function with compact support,
and set $\rho_\eps(x) = \eps^{-1} \rho(\eps^{-1}x)$.
Fix small $\eps > 0$ and set $\varphi_\eps = \rho_\eps* \varphi$.
Then,
we have

\smallskip

\begin{itemize}
\item $\supp \varphi_\eps \subset [-2, 2]$,

\smallskip

\item $\|\varphi_\eps \|_\BMO \sim1$,

\smallskip

\item $\varphi_\eps \in \VMO$ and hence $\varphi_\eps \in \CMO$ (in view of its compact support).

\end{itemize}

Now, given $n \in \N$,
we set $f_n(x) = \|\varphi_\eps \|_\BMO^{-1} \cdot \varphi_\eps(x - 10n)$.
Then, we define a subspace $Y$ of $\BMO$ by
\begin{align*}
 Y ={ \bigg\{
f =  \sum_{n = 1}^\infty a_n f_n : {\bf a} = \{a_n\}_{n \in \N} \in c_0
\bigg\}}.
\end{align*}

\noi
We claim that
that there exist
constants $C_1, C_2 > 0$ such that
\begin{align}
C_1 \sup_{n \in \N} |a_n|
\le \bigg\|\sum_{n = 1}^\infty
a_n f_n \bigg\|_\BMO
\le C_2 \sup_{n \in \N} |a_n|
\label{D4}
\end{align}
\noi
for any ${\bf a} = \{a_n\}_{n \in \N} \in c_0$.
In fact, it follows from
\cite[Proposition 3.1.2 (2)]{G2},
the disjointness of the supports of $f_n$,
and the fact that
the $L^\infty$-norm of $f_n$ is constant in $n \in \N$
that
\[
\bigg\|\sum_{n = 1}^\infty a_n f_n \bigg\|_{\BMO}
\leq 2 \bigg\|\sum_{n = 1}^\infty a_n f_n \bigg\|_{L^\infty} \leq C_2 \sup_{n \in \N} |a_n|.
\]

\noi
On the other hand,  there exists $C_1$, independent of  $k \in \N$,  such that
\[
\bigg\|\sum_{n = 1}^\infty a_n f_n \bigg\|_{\BMO}
\geq  \frac {1}{|\supp f_k|}\int_{\supp f_k} |a_k f_k(x)|\,dx= C_1 |a_k|
\]

\noi
for any $k \in \N$,
 where we used the fact that $f_k$ has mean zero
 in the first step and the second step follows from
noting that the size of the support and the $L^1$-norm of $f_k$ do not depend on $k \in \N$.
This proves~\eqref{D4}.

From the right-hand side of \eqref{D4}, we see that every element in $Y$
is the limit in \BMO of $C_c^\infty$-functions (namely their partial sums).
Hence, from
\cite [Th\'eor\`eme 7 (ii)]{Bour}, we see that  $Y\subset \CMO$.
This shows that $\CMO$
contains a copy of $c_0$.
Hence, from Lemma~\ref{LEM:D2},
we conclude that \CMO
 is not weakly sequentially complete.
By applying Lemma \ref{LEM:B3},
we also see that  $\BMO$ and $\VMO$
are not weakly sequentially complete.
\end{proof}

\smallskip

We conclude this section by recalling that
 \BMO is sequentially complete with respect to the weak-$\ast$ topology as the dual of $H^1$.
 In fact,  this is a particular case of the following general result.
While the proof is standard, we include it for readers' convenience.

\begin{lemma}\label{LEM:A100}
Let $X$ be a Banach space. Then,  $X^*$ is weak-$\ast$ sequentially complete.
\end{lemma}

\begin{proof}
 Let $\{f_n \}_{n \in \N}$
be a weak-$\ast$ Cauchy sequence in $X^*$.
From  the Banach-Steinhaus theorem,
we see that  $\sup_{n \in \N} \|f_n\|_{X^*} < \infty$.
Then, it follows from the Banach-Alaoglu theorem
that there exists a  subsequence
 $\{f_{n_{j}}\}_{j \in \N}$ and  $f \in X^*$ such that, as  $j \to \infty$, $f_{n_j} \to f$ with respect to the weak-$\ast$ topology.
That is,
\[ \jb{f, h} = \lim_{k\to \infty} \jb{f_{n_{j_k}}, h}
= \lim_{n \to \infty} \jb{f_n, h}\]

\noi
for any $h \in X$,
where the second equality follows from the fact that
 $\{f_n \}_{n \in \N}$
is  weak-$\ast$ Cauchy in $X^*$ (which in particular
guarantees uniqueness of the weak-$*$ limit).
 The proof is complete.
\end{proof}

%
%If the above is true then we have another property. If T^*1=0 and T is compact in L^2 then  T:H^1 \to H^1 is compact (and if T1=0 too then by duality from BMO to BMO).
%This would follow from the Riesz transform characterization of H^1, i.e
%
%H1-norm (f) \apporx  L^1-norm(f) + sum L^1-norm(R_j f)
%
%Just note that if T*1=0 then RjT is a Calderon-Zygmun operator (I can give you some references) and since T is compact in L^2 and R_j is bounded then R_jT is compact in L^2 and hence also from H^1 \to L^1.  Then from
%
%H1-norm (Tf) \apporx  L^1-norm(Tf) + sum L^1-norm(R_jT f)
%
%We get the compactness H^1 \to H^1.
%

\begin{ackno}\rm

\'A.B.  acknowledges the support from
 an  AMS-Simons Research Enhancement Grant for PUI Faculty.
G.L. and T.O.~were supported by the European Research Council (grant no.~864138 ``SingStochDispDyn").
The first three authors would like to thank
the West University of Timi\c{s}oara
%Universitatea de Vest din Timi\c{s}oara
for its
hospitality, where part of this paper was prepared.

\end{ackno}

\medskip

\noi
{\bf Declarations of interest.} None.

\end{document}